\documentclass{amsart}
\usepackage{mathrsfs}

\usepackage[colorlinks=true,urlcolor=blue, citecolor=red,linkcolor=blue,
linktocpage,pdfpagelabels, bookmarksnumbered,bookmarksopen]{hyperref}
\usepackage[hyperpageref]{backref}

\usepackage{hyperref}
\usepackage{graphicx}
\usepackage{amssymb, amsmath}
\usepackage{verbatim}
\usepackage{geometry}
\usepackage{esint}
\usepackage{nameref}

\theoremstyle{plain}
\newtheorem{Theorem}{Theorem}[section]
\newtheorem{Proposition}[Theorem]{Proposition}
\newtheorem{Lemma}[Theorem]{Lemma}
\newtheorem{Corollary}[Theorem]{Corollary}
\newtheorem*{main}{Main Theorem}
\newtheorem*{teo2}{Theorem}
\theoremstyle{definition}
\newtheorem{Definition}[Theorem]{Definition}
\newtheorem{Remark}[Theorem]{Remark}
\newtheorem{assumptions}{Assumptions}
\newtheorem*{ack}{Acknowledgments}
\numberwithin{equation}{section}

\date{\today}

\title[Regularity in sets with thick complement]{Regularity for eigenvalue-type equations\\ in sets with thick complement}

\author[Braglia]{Lorenzo Braglia}
\address[L. Braglia]{Dipartimento di Matematica e Informatica
	\newline\indent
	Universit\`a degli Studi di Ferrara
	\newline\indent
	Via Machiavelli 35, 44121 Ferrara, Italy}
\email{lorenzo.braglia@unife.it}

\author[Brasco]{Lorenzo Brasco}
\address[L.\ Brasco]{Dipartimento di Matematica e Informatica
	\newline\indent
	Universit\`a degli Studi di Ferrara
	\newline\indent
	Via Machiavelli 35, 44121 Ferrara, Italy}
\email{lorenzo.brasco@unife.it}


\begin{document}
\begin{abstract}
We prove global higher integrability of the gradient for solutions to nonlinear PDEs with homogeneous Dirichlet condition. We work with general open sets (i.e. not necessarily bounded), under a minimal thickness assumption of the complement. The main focus is on improving the global summability of the gradient, by preserving the ``zero trace" Sobolev class. We pay due attention to providing precise a priori estimates. As an application, we get universal global H\"older estimates for solutions of the super-homogeneous Lane-Emden equation for the $N-$Laplacian, in contractible open sets with finite inradius. 
\end{abstract}

\maketitle

\begin{center}
\begin{minipage}{10cm}
\small
\tableofcontents
\end{minipage}
\end{center}

\section{Introduction}

\subsection{Initial motivation: an old-fashioned result}

Let us consider an open set $\Omega\subsetneq\mathbb{R}^N$, for $0<\alpha\le 1$ we recall the definition of the following space 
\[
C^{0,\alpha}_0(\overline{\Omega})=\Big\{u\in C_{\rm bound}(\overline\Omega)\,:\, [u]_{C^{0,\alpha}(\overline\Omega)}<+\infty,\ u=0\ \text{on}\ \partial\Omega\Big\},
\]
of H\"older continuous functions, vanishing at the boundary.
Here $C_{\rm bound}(\overline\Omega)$ denotes the Banach space of continuous and bounded functions on $\overline\Omega$, while as usual
\[
[u]_{C^{0,\alpha}(\overline\Omega)}=\sup_{x\not=y \in \overline\Omega} \frac{|u(x)-u(y)|}{|x-y|^\alpha}.
\]
We also recall the definition of {\it inradius} of $\Omega$, given by
\[
r_\Omega=\sup\Big\{r>0\, :\, \exists\, B_r(x_0)\subseteq \Omega\Big\}.
\]
The original intent of the present paper was that of proving the following {\it global} regularity result for eigenfunctions of the Dirichlet-Laplacian in the two-dimensional case. We highlight the two main points of interest: the minimality of the assumptions and the universality of the estimate.
\begin{teo2}
There exists a universal exponent $\delta>0$ such that for every simply connected open set $\Omega\subseteq\mathbb{R}^2$ with $r_\Omega<+\infty$ and every eigenfunction $u\in W^{1,2}_0(\Omega)\setminus\{0\}$ of the Dirichlet-Laplacian with eigenvalue $\lambda$, we have 
\[
u\in W^{1,2\,(1+\delta)}_0(\Omega)\qquad \text{and}\qquad u\in C^{0,\frac{\delta}{1+\delta}}_0(\overline\Omega).
\]
Moreover, it holds
\[
\|\nabla u\|_{L^{2\,(1+\delta)}(\Omega)}\le C\,\left(\sqrt{\lambda}\right)^\frac{1+2\,\delta}{1+\delta}\,\|u\|_{L^2(\Omega)},
\]
and
\[
|u|_{C^{0,\frac{\delta}{1+\delta}}(\overline{\Omega})}\le C\,\left(\sqrt{\lambda}\right)^\frac{1+2\,\delta}{1+\delta}\,\|u\|_{L^2(\Omega)},
\]
for a universal constant $C>0$. 
\end{teo2}
We make a couple of comments about the optimality of the previous result. 
\begin{Remark}[A two-dimensional example]
Under the above assumptions, eigenfunctions can not be more than globally H\"older continuous. Indeed, take for example the {\it slit disk}
\[
\Omega_0=\Big\{(x_1,x_2)\in\mathbb{R}^2\,:\, x_1^2+x_2^2<1\Big\}\setminus\Big([0,1]\times \{0\}\Big).
\]
It is not difficult to see that the function given by 
(in polar coordinates)
\[
v(\varrho,\vartheta)=\frac{\sin(\pi\,\varrho)}{\sqrt{\pi\,\varrho}}\,\sin\left(\frac{\vartheta}{2}\right),\qquad \mbox{ for }  0<\varrho<1,\, 0<\vartheta<2\,\pi,
\]
is an eigenfunction (actually, the first one) of the Dirichlet-Laplacian on $\Omega_0$, see for example \cite[Exercise 6.8.9]{TCS}. Observe that for this function we have 
\[
v\in C_0^{0,\frac{1}{2}}(\overline{\Omega_0})\qquad \text{and}\qquad \nabla v\in L^{q}(\Omega_0),\ \text{for every}\ 2\le q<4.
\]
This example also shows that, without additional assumptions, we must have $\delta\le 1$ in the Theorem above.
\end{Remark}
\begin{Remark}[Higher dimensions]
No direct analogue of the above theorem holds in higher dimensions, for the Laplacian. Indeed, for $N\ge 3$ no topological assumptions on $\Omega$ can guarantee a global improvement of regularity. For example, consider the following open set (which is contractible and even starshaped)
\[
\Omega_1=B_1(0)\setminus\Sigma\qquad \text{where}\ \Sigma=\{(t,0,\dots,0)\, :\, t\in[0,1]\},
\]
and use that a segment is {\it removable} for the Laplacian in dimensions $N\ge 3$. Thus, every eigenfunction $u\in W^{1,2}_0(\Omega_1)$ of the Dirichlet-Laplacian on $\Omega_1$ ``does not see'' the boundary part $\Sigma$. In particular, it is not true that 
\[
u\in C^{0,\alpha}_0(\overline{\Omega_1}),
\]
i.e. $u$ does not vanish on $\Sigma$. 
\end{Remark}
As we will see in Corollary \ref{coro:original}, the Theorem above will be a consequence of a much more general statement, obtained by suitably adapting the proof of \cite[Theorem 1.1]{KK}, by Kilpelainen and Koskela. This general statement will be valid for solutions of the quasilinear elliptic equations
\begin{equation}
\label{problem}
-\mathrm{div\,}\nabla G(\nabla u)=f,\qquad \text{in}\ \Omega,
\end{equation}
under suitable assumption on $f$, $G$ and the open set $\Omega$. We insist from the very beginning that our open sets are not necessarily bounded. Moreover, we will try to be as much {\it quantitative} as possible, by focusing on providing precise a priori estimates and discussing explicit classes of open sets to which our result could be applied.

\subsection{Description of the main result}
We start by introducing the class of open sets we wish to consider. They satisfy a sort of ``uniform density condition of the complement'', whose relevance in boundary regularity has been first singled out by \cite{KK}. Nowadays, this is quite common, see for example \cite{BDP, BP, EHL, KSZ} and \cite{Par}, among others. Here ``density" has to be intended in terms of {\it capacity} (we refer to Section \ref{sec:2} for the notation and definition). 
\begin{Definition}
Let $1\le p<\infty$ and let $\Omega\subsetneq\mathbb{R}^N$ be an open set. If there exists $c_0>0$ and $r_0>0$ such that 
\[
\mathrm{cap}_{p}\left(\overline{Q_r(x_0)}\setminus\Omega;Q_{2\,r}(x_0)\right)\ge c_0\,\mathrm{cap}_{p}\left(\overline{Q_r(x_0)};Q_{2\,r}(x_0)\right),\quad \text{for every}\ x_0\in\mathbb{R}^N\setminus\Omega,\ 0<r\le r_0,
\]
then we say that $\mathbb{R}^N\setminus\Omega$ is {\it uniformly $p-$thick with constants $(c_0,r_0)$}. We denote the collection of all open sets with this property by
\[
\mathrm{UT}_{p}(c_0,r_0).
\]
If the previous estimate holds for every $r>0$, we will simply say that $\mathbb{R}^N\setminus\Omega$ is {\it uniformly $p-$thick with constant $c_0$}. Accordingly, we will use the symbol
\[
\mathrm{UT}_{p}(c_0),
\] 
for the collection of all open sets having this property.
\end{Definition}
In what follows, for $1<p<\infty$ we will use the following notation
\begin{equation}
\label{lambda}
\lambda_{p}(\Omega)=\inf_{\varphi\in C^\infty_0(\Omega)}\left\{\int_\Omega |\nabla \varphi|^p\,dx\, :\, \|\varphi\|_{L^p(\Omega)}=1\right\}.
\end{equation}
Observe that we have
\[
\lambda_{p}(\Omega)>0\qquad \Longleftrightarrow\qquad \mathscr{D}^{1,p}_0(\Omega)\hookrightarrow L^p(\Omega),
\]
and it holds by definition
\[
\lambda_{p}(\Omega)\,\int_\Omega |\varphi|^p\,dx\le \int_\Omega |\nabla \varphi|^p\,dx,\qquad \text{for every}\ \varphi\in C^\infty_0(\Omega).
\]
The symbol $\mathscr{D}^{1,p}_0(\Omega)$ stands for the homogeneous Sobolev space, given by the completion of $C^\infty_0(\Omega)$ with respect to the norm
\[
\varphi\mapsto \|\nabla \varphi\|_{L^p(\Omega)}.
\]
As already anticipated, our main result is concerned with weak solutions to the equation \eqref{problem}. We will work under the following list of assumptions on the open set $\Omega\subseteq\mathbb{R}^N$, the right-hand side $f$ and the function $G$. 
\begin{assumptions}
\label{ass:1}
We will assume that:
\begin{itemize}
\item[($E_f$)] \makeatletter\def\@currentlabel{$E_f$}\makeatother\label{item:Ef} $f\in L^{\gamma_0}(\Omega)$, with $\gamma_0\le p'$ such that
\[
\left\{\begin{array}{rl}
(p^*)'\le \gamma_0& \text{if}\ 1<p<N,\\
1<\gamma_0,& \text{if}\ p=N,\\
1\le \gamma_0,& \text{if}\ p>N;
\end{array}
\right.\qquad \text{where}\ p^*=\frac{N\,p}{N-p};
\]
\item[($E_\Omega$)] \makeatletter\def\@currentlabel{$E_\Omega$}\makeatother\label{item:EOmega} $\Omega\subseteq \mathbb{R}^N$ is an open set such that $\lambda_{p}(\Omega)>0$;
\vskip.2cm
\item[($E_G$)] \makeatletter\def\@currentlabel{$E_G$}\makeatother\label{item:EG} $G:\mathbb{R}^N\to\mathbb{R}$ is a $C^1$ strictly convex function, such that 
\[
g_1\,|z|^p\le G(z)\le g_2\,|z|^p,
\]
 for two constants $g_2\ge g_1>0$. 
\end{itemize}
\end{assumptions}
\begin{assumptions}
\label{ass:2}
In order to improve the regularity of solutions, we will take the following further set of assumptions:
\begin{itemize}
\item[($R_f$)] \makeatletter\def\@currentlabel{$R_f$}\makeatother\label{item:Rf} $f\in L^\gamma(\Omega)$ for 
\[
\gamma_0<\gamma\le +\infty,
\]
and $\gamma_0$ as in Assumptions \ref{ass:1};
\vskip.2cm
\item[($R_\Omega$)] \makeatletter\def\@currentlabel{$R_\Omega$}\makeatother\label{item:ROmega} $\Omega\in \mathrm{UT}_{p_0}(c_0)$, for some $1\le p_0<p$ and $c_0>0$.
\vskip.2cm
\end{itemize}
\end{assumptions}
 We are now ready to state the main result of this paper. We will denote by $W^{1,p}_0(\Omega)$ the closure of $C^\infty_0(\Omega)$ in the usual Sobolev space $W^{1,p}(\Omega)$.
\begin{main}
Under the Assumptions \ref{ass:1} and \ref{ass:2}, let $u\in W_0^{1,p}(\Omega)$ be the weak solution to the equation \eqref{problem}.
Then, there exists $\delta=\delta(N,p,p_0,\gamma_0,\gamma,c_0,g_1,g_2)>0$ with $\delta\le (\gamma-\gamma_0)/\gamma_0$,
such that 
\[
u\in W_0^{1,p\,(1+\delta)}(\Omega).
\] 
Moreover, the following scale invariant estimate holds, for $C_1=C_1(N,p,p_0,\gamma_0,\gamma,c_0,g_1,g_2)>0$
\[
\begin{split}
\|\nabla u\|_{L^{p\,(1+\delta)}(\Omega)}^p&\le C_1\,\left[\int_\Omega |\nabla u|^p\,dx\right]^{1-\Theta}\,\left[\left(\|f\|_{L^{\gamma_0}(\Omega)}\right)^{p'-\frac{\delta}{1+\delta}\,\frac{\gamma\,\gamma_0}{\gamma-\gamma_0}}\,\left(\|f\|_{L^{\gamma}(\Omega)}\right)^{\frac{\delta}{1+\delta}\,\frac{\gamma\,\gamma_0}{\gamma-\gamma_0}}\right]^\Theta,
\end{split}
\]
where $\Theta=\Theta(N,p,\gamma_0,\delta)\in(0,1)$ is given by 
\[
\frac{1}{\Theta}=1+\frac{1+\delta}{\delta}\,p'\,\left(\frac{1}{\gamma_0'}-\frac{N-p}{N\,p}\right).
\]
\end{main}

\subsection{Comments on the assumptions}

We can not expect the reader to immediately understand both the necessity and the subtleties hidden in the (long) list of assumptions we made. Of course, some of them are needed in order to consider unbounded open sets. In this subsection, we want to thoroughly comment these assumptions.
\begin{itemize}
\item {\it The operator:} differently from \cite{KK}, our equation \eqref{problem} has a variational nature. In other words, it is the Euler-Lagrange equation of the (strictly) convex functional
\begin{equation}
\label{funzionale}
\mathfrak{F}(\varphi)=\int_\Omega G(\nabla \varphi)\,dx-\int_\Omega f\,\varphi\,dx,\qquad \text{for every}\ \varphi\in W^{1,p}_0(\Omega).
\end{equation}
Accordingly, the solution $u$ is the minimizer of this functional. This fact is exploited in order to run a certain approximation argument (see Section \ref{sec:4}), that we will explain in a moment. 
In doing this, the minimality of $u$ will be useful.
\vskip.2cm
\item {\it The function $G$:} if $G$ is not differentiable, the equation \eqref{problem} would not make sense. However, in this case it would be possible to get the Main Theorem for the minimizer $u$ of \eqref{funzionale}. Indeed, the desired Caccioppoli inequality, which is the cornerstone of the result, could be obtained in this case by using the classical trick of Giaquinta and Giusti (see \cite{GG}), circumventing the lack of an Euler-Lagrange equation. This is the case treated by Granlund in \cite[Theorem 1.4]{Gr}, for example, under stronger assumptions on the open set. Here, we preferred to keep the $C^1$ assumption, in order to simplify a bit some of the estimates.
\vskip.2cm
\item {\it The assumptions {\rm(\ref{item:Ef})}  and {\rm(\ref{item:Rf})} on the right-hand side $f$}: since we are not assuming that $\Omega$ has finite volume, the two integrability requirements $f\in  L^{\gamma_0}(\Omega)$ and $f\in L^{\gamma}(\Omega)$ are not redundant. Both of them play a role: the former, in conjunction with (\ref{item:EOmega}), is needed in order to assure {\it existence} of a solution; the latter is responsible for the {\it gain of integrability}.  We also point out the requirement $\gamma_0\le p'$: this is crucially needed to get a local improvement of integrability, which can be extended to the whole $\Omega$ by using a covering with cubes having {\it fixed} radius. We refer to {\it Step 5} and {\it Step 6} of the proof of the Main Theorem.
\par
 If $\Omega$ has finite volume, this restriction is immaterial, since for $\gamma_0>p'$ we would have $f\in L^{\gamma_0}(\Omega)\subseteq L^{p'}(\Omega)$.
However, when $\Omega$ has infinite volume and $\gamma_0>p'$, it seems that a finer knowledge on the ``geometry at infinity'' of the open set $\Omega$ would be needed. We plan to come back on this issue in the future.
\vskip.2cm
\item {\it The assumption {\rm (\ref{item:EOmega})}  on the open set $\Omega$:} we notice that the assumption $\lambda_{p}(\Omega)>0$ guarantees the continuity of the {\it super-homogeneous} embeddings
\[
\mathscr{D}^{1,p}_0(\Omega)\hookrightarrow L^{\gamma_0'}(\Omega),\qquad \mbox{for every}\ p\le \gamma_0' \left\{\begin{array}{rl}
\le p^*& \text{if}\ 1<p<N,\\
<\infty,& \text{if}\ p=N,\\
\le \infty,& \text{if}\ p>N,
\end{array}
\right.
\]
and the identification $\mathscr{D}^{1,p}_0(\Omega)=W^{1,p}_0(\Omega)$ (see Proposition \ref{prop:inradius_finito}).
These embeddings are perfectly consistent with the assumption {\rm(\ref{item:Ef})} on $f$
and they assure that we have
\[
f\,\varphi\in L^1(\Omega),\qquad \mbox{for every}\ \varphi\in W^{1,p}_0(\Omega).
\] 
We recall that the assumption $\lambda_p(\Omega)>0$ comes for free if $\Omega$ has finite volume, as it is well-known. In general, necessary and sufficient conditions for this condition to hold are contained in \cite[Chapter 15]{Maz}.
\par
We point out that the case $1<p<N$ and $\gamma_0=(p^*)'$ is somehow exceptional: we could get a result similar to that of the Main Theorem by deleting (\ref{item:EOmega}) and only requiring the thickness condition (\ref{item:ROmega}). In this case, one should work with the homogeneous Sobolev space $\mathscr{D}^{1,p}_0(\Omega)$, in place of $W^{1,p}_0(\Omega)$ (we briefly report this variant in Theorem \ref{maintheorem} in Appendix \ref{app:B}, for completeness).
\par
However, for $p\ge N$ this would cause some serious issues: the lack of a suitable surrogate for the Sobolev inequality makes the space $\mathscr{D}^{1,p}_0(\Omega)$ quite nasty to work with. In general, this is not even a functional space (see for example \cite{BGCV, DL} or \cite[Chapter 15]{Maz}). This explains our choice of taking the assumption (\ref{item:EOmega});
\vskip.2cm
\item {\it The thickness condition {\rm (\ref{item:ROmega})}:} the attentive reader could notice that our assumption on the thickness of $\mathbb{R}^N\setminus\Omega$ is apparently stronger than that taken in \cite{KK}. Actually, this is not exactly true: in \cite{KK} it is assumed that $\mathbb{R}^N\setminus\Omega$ is uniformly $p-$thick, but then this condition is used to obtain that the same property holds for {\it some} $p_0<p$, in light of a sophisticated result by Lewis (see \cite[Theorem 1]{Lewis}). In the proof of \cite{KK}, it is really the $p_0-$thickness condition which is used. Accordingly, all the estimates in \cite{KK} depends on this exponent $p_0$, which is known to exist but is quite implicit. For this reason, we decided to directly assume the condition on some exponent $p_0<p$, in order to have a better control on the a priori estimate;
\vskip.2cm
\item {\it The thickness condition {\rm (\ref{item:ROmega})}...reprise:} the same attentive reader could also notice that in \cite{KK} it is assumed that $\Omega\in \mathrm{UT}_{p}(c_0,r_0)$, for some {\it finite} $r_0>0$. This requirement is apparently less restrictive than ours. Again, we show in Appendix \ref{app:A} that this is not the case: the assumption (\ref{item:EOmega})  automatically entails that if $\mathbb{R}^N\setminus\Omega$ is uniformly $p-$thick for a range of radius $0<r\le r_0$, then it actually has the same property {\it for every $r>0$} (possibly by spoiling the constant $c_0$, in a controlled way). This is a straightforward consequence of our Theorem \ref{thm:salvaipotesi}, which we believe is interesting in itself. Observe that the sets considered in \cite{KK} were {\it bounded}, thus the assumption (\ref{item:EOmega})  was automatically satisfied.
\end{itemize}

\subsection{Comments on the proof}

Our strategy of proof is exactly the same as in the paper by Kilpelainen and Koskela \cite{KK}, in turn inspired by Granlund's paper \cite{Gr}. The idea is to join the Caccioppoli and Sobolev-Poincar\'e inequalities, in order to obtain an ``unnatural'' reverse H\"older inequality on cubes, for the gradient of the solution. By iterating the estimate at every scale, a suitable variant of Gehring's Lemma implies the desired integrability gain on the gradient. This is super classical nowadays, as far as {\it local interior} estimates are concerned. But the main point here and in \cite{Gr, KK} is to get the estimate up to the boundary.
Indeed, the crucial point is the following one: for cubes intersecting $\mathbb{R}^N\setminus \Omega$, it is possible to use a capacitary variant of the Sobolev-Poincar\'e inequality due to Maz'ya, exactly because of the thickness condition.
\par
However, some care is needed. This procedure would only lead to the global integrability improvement
\[
\nabla u \in L^{p\,(1+\delta)}(\Omega),
\]
for some $\delta>0$ depending on the data, as in \cite{KK}. In other words, this integrability improvement would not keep the\footnote{The reader should not be fooled by our informal use of the word ``trace''. Actually, in this paper we are not using any trace theory and our functions ``vanish at the boundary'' in the weak sense prescribed by the membership to the closure space $W^{1,p}_0$.} ``zero trace'' Sobolev class, i.e. one could hope for the (slightly) stronger result
\[
u\in W^{1,p\,(1+\delta)}_0(\Omega).
\]
This slight improvement would be more natural, since it would be a true result ``up to the boundary'', in a sense, as it really keeps track of the boundary datum. But more important, such an improvement would be particularly interesting in the borderline case $p=N$, since the Sobolev-Morrey--type embedding into globally H\"older continuous functions
\[
W^{1,N\,(1+\delta)}_0(\Omega)\hookrightarrow C^{0,\frac{\delta}{1+\delta}}_0(\overline\Omega),
\]
holds {\it for every open set}, on the scale of $W^{1,p}_0$ spaces, differently from the case of $W^{1,p}$. 
\par
We thus devote the rest of this subsection to explain how we get the higher integrability with ``zero trace'':
\begin{itemize}
    \item 
 at first, one can observe that it is sufficient to find a sequence $\{u_n\}_{n\in\mathbb{N}}\subseteq W_0^{1,p\,(1+\delta)}(\Omega)$, for some $\delta$ independent on $n$, such that
    \begin{equation}\label{closure}
        u_n\,\rightharpoonup\, u,\qquad\text{in}\ W^{1,p\,(1+\delta)}(\Omega),\qquad\text{as}\ n\to\infty,
    \end{equation}
  thanks to the fact that  $W_0^{1,p\,(1+\delta)}(\Omega)$ is closed both in the strong and weak topologies (see for example \cite[Theorem 3.7.4]{BraBook});
  \vskip.2cm 
    \item we construct the sequence $\{u_n\}_{n\in\mathbb{N}}$ by variational methods: for every $n\in\mathbb{N}$, we take the minimizer of the following ``fake'' double-phase functional with $1<p<q$
\[
\varphi\mapsto \int_\Omega G(\nabla \varphi)\,dx+\frac{\alpha_n}{q}\,\int_\Omega |\nabla  \varphi|^q\,dx-\int_\Omega f\,\varphi\,dx,
\]
where $q>p$ is fixed and $\{\alpha_n\}_{n\in\mathbb{N}}$ is a sequence of {\it positive} real numbers, converging to $0$ as $n$ goes to $\infty$. This functional is well-defined on $W^{1,p}_0(\Omega)\cap W^{1,q}_0(\Omega)$. Thus, by a simple interpolation argument each $u_n$ lies in particular in $W^{1,t}_0(\Omega)$, for every $p\le t\le q$. Moreover, it is not difficult to show that
\[
u_n\,\rightharpoonup\, u,\qquad\text{in}\ W^{1,p}(\Omega),\qquad\text{as}\ n\to\infty;
\]
\item thanks to our hypothesis on $\Omega$, we have that for every $p\le t\le q$ 
\[
\varphi\mapsto \|\nabla \varphi\|_{L^t(\Omega)},
\]
is an equivalent norm on $W^{1,t}_0(\Omega)$. Therefore, in order to improve the previous weak convergence and get \eqref{closure}, it is  sufficient to get
\[
\nabla u_n\in L^{p\,(1+\delta)}(\Omega),\qquad \text{for every}\ n\in\mathbb{N},
\]
with a uniform $\delta$ and a uniform control on the relevant norms. Indeed, recall that $u_n\in W^{1,p\,(1+\delta)}_0(\Omega)$ just by construction (at least for $\delta$ small enough); 
\vskip.2cm 
\item the previous global higher integrability result can be obtained by repeating the above scheme of proof: joining Caccioppoli and Maz'ya-Sobolev-Poincar\'e inequalities to get a reverse H\"older inequality. This goes exactly as in \cite{KK}, up to some technical modifications needed to cope with the presence of the right-hand side $f$ and the possible unboundedness of $\Omega$.
\par
At a technical level, we proceed as follows: we get uniform reverse H\"older inequalities not for $|\nabla u_n|^p$, but for the following non-homogeneous quantity
\[
H_n(|\nabla u_n|):=|\nabla u_n|^p+\alpha_n\,|\nabla u_n|^q,
\]
which is naturally attached to the perturbed problem.
The idea of considering this quantity is taken from \cite[Theorem 1.1]{CM1}, which deals with regularity issues for double-phase problems (see also \cite{BCM, CM2, DeMi}). The above described scheme will then produce the following estimate
\[
\int_\Omega H_n(|\nabla u_n|)^{1+\delta}\le C,
\]
with $C$ and $\delta$ uniform in $n$, implying the same uniform bound for the $L^{p\,(1+\delta)}(\Omega)$ norm of $\nabla u_n$, as claimed.
\end{itemize}

\begin{Remark}
After the completion of the paper, we became aware of the reference \cite{HeKi} by Hedberg and Kilpelainen. In particular, \cite[Corollary 3.5]{HeKi} shows that 
\[
W^{1,p\,(1+\delta)}(\Omega)\cap W^{1,p}_0(\Omega)=W^{1,p\,(1+\delta)}_0(\Omega),
\]
whenever $\Omega\subseteq\mathbb{R}^N$ is an open bounded set, whose complement $\mathbb{R}^N\setminus\Omega$ is uniformly $p-$thick. Apart for being stated for {\it bounded} sets, the proof of their result is based on the fine properties of Sobolev functions, i.e. it uses $p-$quasi continuous representatives of Sobolev functions and the sophisticated {\it Havin-Bagby Theorem} for characterizing $W^{1,p}_0$ spaces on open sets. 
\par
We believe that our proof, based on constructing in an elementary way an approximating sequence and relying only on very basic properties of capacities and Sobolev spaces, is interesting in itself. The argument is simple and robust enough to be quite easily extended to more general situations (i.e. double-phase operators, $p(x)-$Laplacian, orthotropic equations with nonstandard growth, just to list some examples), without the necessity to exploit the fine properties of the underlying Sobolev spaces.
\end{Remark}

\subsection{Plan of the paper}
We start with Section \ref{sec:2}, containing all the notation and definitions needed throughout the paper. We also collect some properties of open sets with thick complement, which will be needed for the proof of the Main Theorem. In Section \ref{sec:3} we construct the approximating sequence $\{u_n\}_{n\in\mathbb{N}}$, through the ``fake'' double-phase functional previously discussed. We also obtain the two main Caccioppoli inequalities for these minimizers, needed to get the desired reverse H\"older inequalities. The proof of the Main Theorem is contained in Section \ref{sec:4}: it will be divided in various steps, in order to assist the reader. We will get the result at first under a normalization condition; then we will use a scaling argument to reach the desired scale invariant estimate. The main purpose of Section \ref{sec:5} is that of offering a simple and sufficiently rich class of open sets to which our main result applies: we will show that {\it contractible} open sets satisfy the thickness conditions, under suitable restrictions on $p$ and the dimension. We will also comment on the geometric aspects of the assumption {\rm (\ref{item:EOmega})}, for this class of sets (see Proposition \ref{prop:geometric}). Finally, in Section \ref{sec:6} we will come to our original intent and show how our Main Theorem can be applied to the case of Lane-Emden equations (including the eigenvalue equation).
\par
The paper is complemented by two appendices: while Appendix \ref{app:B} is very brief and simply shows how the Main Theorem could be further extended in the case $1<p<N$, in Appendix \ref{app:A} we prove a property of open sets in the class $\mathrm{UT}_p(c_0,r_0)$, that we believe to be of independent interest (see Theorem \ref{thm:salvaipotesi} below). 

\begin{ack}
We wish to thank Paolo Baroni, Roberto Ognibene and Mikko Parviainen for some useful conversations.
We thank Andrea Cianchi for suggesting to us how to get the homogeneous Morrey's inequality on the sphere, needed in the proof of Lemma \ref{lm:poincareS}. 
\par
L. Braglia has been supported by the joint Ph.D. program of the Universities of Ferrara, Modena \& Reggio-Emilia and Parma.
\par
L.\,Brasco has been financially supported by the {\it Fondo di Ateneo per la Ricerca} FAR 2024 and the {\it Fondo per l'Incentivazione alla Ricerca Dipartimentale} FIRD 2025 of the University of Ferrara.
\par
Part of this research has been done during the conference ``{\it (Ir)Regularity $@$Parma I\,}'', held in Parma in May 2026. We wish to thank the organizers and the supporting institutions for the kind invitation and the nice working atmosphere provided.
\end{ack}

\section{Preliminaries}
\label{sec:2}

\subsection{Notation} 
Throughout the whole paper, we will always assume that the dimension $N\ge 2$.
We denote by $Q_r(x_0)$ the open cube centered at $x_0\in\mathbb{R}^N$ with side length $2\,r>0$, that is
\[
Q_r(x_0)=\{x\in\mathbb{R}^N\, :\, \|x-x_0\|_{\ell^\infty}<r\},
\]
 and by $B_r(x_0)$ the open ball centered at $x_0\in\mathbb{R}^N$ with radius $r>0$. For a measurable set $E\subseteq \mathbb{R}^N$ such that $0<|E|<+\infty$ we denote by
\[
 \overline{u}_E=\fint_E u\,dx= \frac{1}{|E|}\,\int u\,dx,
\]
the averaged integral of the summable function $u$ over $E$. The symbol $|E|$ stands for the $N-$dimensional Lebesgue measure of $E$. 
\subsection{Basics}
The next result collects some properties of the open sets satisfying Assumptions \ref{ass:1}. Recall that the quantity $\lambda_p(\Omega)$ has been defined in \eqref{lambda}.
\begin{Proposition}
\label{prop:inradius_finito}
Let $1<p<\infty$ and let $\gamma_0\le p'$ be such that
\[
\left\{\begin{array}{rl}
(p^*)'\le \gamma_0& \text{if}\ 1<p<N,\\
1<\gamma_0,& \text{if}\ p=N,\\
1\le \gamma_0,& \text{if}\ p>N.
\end{array}
\right.
\]
Let $\Omega \subseteq \mathbb{R}^N$ be an open set such that $\lambda_{p}(\Omega)>0$, then:
\begin{enumerate}
\item $\Omega$ has finite inradius, i.e. $r_\Omega<+\infty$;
\vskip.2cm 
\item $\mathscr{D}^{1,p}_0(\Omega)=W^{1,p}_0(\Omega)$ and
\[
\varphi \mapsto \|\nabla \varphi\|_{L^p(\Omega)},
\]
is an equivalent norm on $W^{1,p}_0(\Omega)$;
\vskip.2cm
\item we also have 
\begin{equation}
\label{costanti}
\lambda_{p,\gamma_0'}(\Omega):=\inf_{\varphi\in C^\infty_0(\Omega)}\left\{\int_\Omega |\nabla \varphi|^p\,dx\, :\, \|\varphi\|_{L^{\gamma_0'}(\Omega)}=1\right\}>0.
\end{equation}
\end{enumerate}
\end{Proposition}
\begin{proof}
The first result is a well-known necessary condition for the validity of Poincar\'e inequalities, see for example \cite[Proposition 2.1]{Sou}. 
\par
The point (2) easily follows from the fact that $\lambda_p(\Omega)>0$, since this is equivalent to the validity of the Poincar\'e inequality
\[
\lambda_p(\Omega)\,\int_\Omega |\varphi|^p\,dx\le \int_\Omega |\nabla \varphi|^p\,dx,\qquad \text{for every}\ \varphi\in C^\infty_0(\Omega),
\]
as already said.
\par
As for the last point, we can suppose $\gamma_0<p'$. Then we have
\[
\left\{\begin{array}{rl}
p<\gamma_0'\le p^*& \text{if}\ 1<p<N,\\
p<\gamma_0'<\infty,& \text{if}\ p=N,\\
p< \gamma_0'\le \infty,& \text{if}\ p>N.
\end{array}
\right.
\]
Observe that for the limit case $\gamma_0'=p^*$ we always have
\[
\lambda_{p,p^*}(\Omega)>0.
\] 
Indeed, this constant coincides with the sharp Sobolev constant, which is independent of the open set $\Omega$ (see \cite[Chapter I, Section 4.5]{St}).
For the remaining case, we recall that \cite[Theorem 15.4.1]{Maz} shows
\[
\lambda_p(\Omega)>0\qquad \Longleftrightarrow\qquad \lambda_{p,{\gamma_0'}}(\Omega)>0. 
\]
This concludes the proof.
\end{proof}
The following simple result is concerned with the functions $G$ considered in the Main Theorem. This should be well-known, we include it for completeness.
\begin{Lemma}
\label{lm:derivataG}
Let $1<p<\infty$ and let $G:\mathbb{R}^N\to\mathbb{R}$ be a $C^1$ convex function, such that 
\[
g_1\,|z|^p\le G(z)\le g_2\,|z|^p,
\]
for two constants $g_2\ge g_1>0$. Then we have
\[
|\nabla G(z)|\le (2^p\,g_2-g_1)\,|z|^{p-1},\qquad \text{for every}\ z\in\mathbb{R}^N,
\]
and
\[
\langle\nabla G(z),z\rangle\ge g_1\,|z|^p,\qquad \text{for every}\ z\in\mathbb{R}^N.
\]
\end{Lemma}
\begin{proof}
We take $z\in\mathbb{R}^N$ and $\omega\in\mathbb{S}^{N-1}$. By the ``above tangent'' property of convex functions, we have
\[
G(z+|z|\,\omega)\ge G(z)+\langle \nabla G(z),\omega\rangle\,|z|.
\]
In particular, by using the growth conditions on $G$, we obtain
\[
g_2\,\big|z+|z|\,\omega\big|^p\ge g_1\,|z|^p+\langle \nabla G(z),\omega\rangle\,|z|.
\]
By the triangle inequality, we can further estimate the left-hand side as follows
\[
\big|z+|z|\,\omega\big|^p\le 2^p\,|z|^p.
\]
Thus, in conclusion we obtain
\[
\langle \nabla G(z),\omega\rangle\,|z|\le (2^p\,g_2-g_1)\,|z|^p.
\]
By taking the maximum over $\omega\in\mathbb{S}^{N-1}$, we get the first estimate.
\par
For the second one, observe that for $z=0$ there is nothing to prove. We take $z\not=0$ and we observe that for every $t\in[0,1]$ we have
\[
\langle \nabla G(z)-\nabla G(t\,z),z-z\,t\rangle\ge0.
\]
This follows from the monotonicity of $\nabla G$, which in turn is a consequence of the convexity of $G$. In particular, the previous inequality implies that
\begin{equation}
\label{monotona}
\langle \nabla G(z),z\rangle\ge \langle \nabla G(t\,z),z\rangle,\qquad \text{for every}\ t\in[0,1].
\end{equation}
We can now write
\[
g_1\,|z|^p\le G(z)=G(z)-G(0)=\int_0^1 \langle \nabla G(t\,z),z\rangle\,dt\le \langle \nabla G(z),z\rangle,
\]
thanks to \eqref{monotona}. This gives the second estimate, as well.
\end{proof}
\vskip.2cm\noindent
As in \cite{KK}, we will make a crucial use of the following classical result, the so-called {\it Gehring Lemma} (see \cite{Ge}) in the version of Giaquinta and Modica \cite{GM} and Stredulinsky \cite{Str}. Differently from \cite{KK}, we prefer to work here with cubes, instead of balls.
The statement is taken from
\cite[Theorem 6.6]{Gi}, to which we refer for the proof.
\begin{Lemma}
\label{lm:gehringG}
Let $h\in L^1(Q_L(x_0))$ and $k\in L^s(Q_L(x_0))$, for some $s>1$. Let us suppose that there exists $0<\varepsilon<1$ and $C>0$ such that for every $Q_r(y_0)\Subset Q_L(x_0)$ we have \begin{equation}
\label{gehrip}
\fint_{Q_{r/2}(y_0)} |h|\,dx\le C\,\left[\left(\fint_{Q_r(y_0)}|h|^{(1-\varepsilon)}\,dx\right)^{\frac{1}{1-\varepsilon}}+\fint_{Q_r(y_0)}|k|\,dx\right].
\end{equation}
Then, there exists an exponent $\delta=\delta(N,\varepsilon,C,s)>0$ such that $\delta\le \min\{1,s-1\}$ and a constant $\widetilde{C}>0$ such that 
\[
h\in L^{1+\delta}(Q_{L/2}(x_0)),
\] 
with the following estimate
 \begin{equation}\label{gehrth}
     \fint_{Q_{L/2}(x_0)} |h|^{1+\delta}\,dx\le \widetilde{C}\,\left[\left(\fint_{Q_L(x_0)}|h|\,dx\right)^{1+\delta}+\fint_{Q_L(x_0)}|k|^{1+\delta}\,dx\right].
 \end{equation}
\end{Lemma}

\subsection{Open sets with thick complement}

In the proof of the Main Theorem we will exploit some properties of {\it variational $p-$capacity}. 
We recall at first the definition employed in this paper.
\begin{Definition}
Let $1\le p<\infty$, for every $E\subseteq\mathbb{R}^N$ open set and every $K\subseteq E$ compact set, the {\it $p-$capacity of $K$ relative to $E$} is defined as
\[
\mathrm{cap}_p(K;E)=\inf_{\varphi\in C^\infty_0(E)}\left\{\int_E |\nabla \varphi|^p\,dx\, :\, \varphi \ge 1\ \text{on}\ K\right\}.
\]
We refer to \cite{Maz} for a comprehensive study on the properties of the $p-$capacity.
\end{Definition}
We observe that the notion of uniform $p-$thickness is usually stated in terms of balls. It is not difficult to show that this is equivalent to our definition in terms of cubes. This is a straightforward consequence of the next result, that we enclose for completeness.
\begin{Lemma}
\label{equivalencecubesballs}
Let $1\le p<\infty$ and let $K\subseteq\mathbb{R}^N$ be a non-empty compact set. There exist two constants $0<\beta_1<\beta_2$, both depending on $N$ and $p$ only, such that
for every $x\in K$ and every $r>0$ we have that
 \begin{equation*}
 \beta_2\,   \mathrm{cap}_p\left(\overline{Q_r(x)}\cap K,Q_{2r}(x)\right)\ge    \mathrm{cap}_p\left(\overline{B_r(x)}\cap K,B_{2r}(x)\right)\ge \beta_1\,   \mathrm{cap}_p\left(\overline{Q_r(x)}\cap K,Q_{2r}(x)\right).
 \end{equation*}
 \end{Lemma}
\begin{proof}
We start by observing that $B_r(x)\subseteq  Q_r(x)\subseteq  B_{r\sqrt{N}}(x)$. 
By the monotonicity properties of the relative capacity, we infer that
\begin{equation*}
      \mathrm{cap}_p\left(\overline{Q_r(x)}\cap K,Q_{2r}(x)\right)\ge
      \mathrm{cap}_p\left(\overline{B_r(x)}\cap K,Q_{2r}(x)\right)\ge  \mathrm{cap}_p\left(\overline{B_r(x)}\cap K,B_{2r\sqrt{N}}(x)\right).
\end{equation*}
We can appeal again to \cite[Proposition 13.1.1.1]{Maz}, in order to get
\[
\mathrm{cap}_p\left(\overline{B_r(x)}\cap K,B_{2r\sqrt{N}}(x)\right)\ge c\,\mathrm{cap}_p\left(\overline{B_r(x)}\cap K,B_{2r}(x)\right),
\]
for a constant $c=c(N,p)>0$. Thus, we get
\[
\mathrm{cap}_p\left(\overline{Q_r(x)}\cap K,Q_{2r}(x)\right)\ge c\,\mathrm{cap}_p\left(\overline{B_r(x)}\cap K,B_{2r}(x)\right),
\]
i.e., the leftmost estimate.
The proof of the other estimate is similar, it is sufficient to notice that $Q_{r/\sqrt{N}}(x)\subseteq  B_r(x)\subseteq  Q_r(x)$.
\end{proof}

The next basic result is quite straightforward (see also \cite[Remark 4]{BDP}), we give the details in order to keep track of the constants.
\begin{Lemma}
\label{lm:scendicapacity}
Let $1\le p_0<\infty$ and let $\Omega\in\mathrm{UT}_{p_0}(c_0,r_0)$, for some $c_0,r_0>0$. For every $q>p_0$ there exists a constant $c_1=c_1(N,p_0,q)>0$ such that 
\[
\Omega\in\mathrm{UT}_{q}(c_1,r_0),
\]
as well.
\end{Lemma}
\begin{proof}
We take $x_0\in\mathbb{R}^N\setminus\Omega$ and $0<r\le r_0$. 
For every $\varphi\in C^\infty_0(Q_{2r}(x_0))$ such that $\varphi\ge 1$ on $\overline{Q_{r}(x_0)}\setminus\Omega$, we have
\[
\begin{split}
c_0\,\mathrm{cap}_{p_0}\left(\overline{Q_r(x_0)};Q_{2\,r}(x_0)\right)&\le \mathrm{cap}_{p_0}(\overline{Q_{r}(x_0)}\setminus\Omega;Q_{2r}(x_0))\\
&\le \int_{Q_{2r}(x_0)} |\nabla \varphi|^{p_0}\,dx\le |Q_{2r}(x_0)|^{1-\frac{p_0}{q}}\,\left(\int_{Q_{2r}(x_0)} |\nabla \varphi|^{q}\,dx\right)^\frac{p_0}{q}.
\end{split}
\]
By taking the infimum over $\varphi$, we get the desired conclusion.
\end{proof}
\begin{Remark}
By inspecting the previous proof, we see that we can choose $c_1$ to be
\[
c_1=4^N\,\left(\frac{c_0}{4^N}\right)^\frac{q}{p_0}\,\frac{\left(\mathrm{cap}_{p_0}\left(\overline{Q_1};Q_{2}\right)\right)^\frac{q}{p_0}}{\mathrm{cap}_{q}\left(\overline{Q_1};Q_{2}\right)}.
\]
\end{Remark}
The following simple result expresses the scaling properties of the thickness condition. This will be useful in order to get the a priori estimate of the main result.
\begin{Lemma}
\label{lm:riscalothick}
Let $1\le p_0<\infty$ and let $\Omega \subsetneq  \mathbb{R}^N$ be an open set such that $\Omega\in \mathrm{UT}_{p_0}(c_0,r_0)$. For every $t>0$, the scaled set $t\,\Omega$ satisfies
\[
t\,\Omega \in \mathrm{UT}_{p_0}(c_0,t\,r_0).
\] 
\end{Lemma}
\begin{proof}
Let $x_0\in \mathbb{R}^N \setminus (t\,\Omega)$, then $x_0/t\in \mathbb{R}^N\setminus\Omega$. Accordingly, for every $0<r\le t\,r_0$ we have 
\begin{equation}
\label{flup!}
\begin{split}
\mathrm{cap}_{p_0}\left(\overline{Q_{r/t}(x_0/t)}\setminus\Omega;Q_{2r/t}(x_0/t)\right)&\ge c_0\,\mathrm{cap}_{p_0}\left(\overline{Q_{r/t}};Q_{2r/t}\right)\\
&=c_0\,\mathrm{cap}_{p_0}\left(\overline{Q_{r}};Q_{2r}\right)\,t^{p_0-N}.
\end{split}
\end{equation}
We used both the definition of $p_0-$thickness and the scaling properties of the relative $p_0-$capacity. On the other hand, we have that
\[
\begin{split}
t^{p_0-N}\,\mathrm{cap}_{p_0}\left(\overline{Q_r(x_0)}\setminus(t\,\Omega);Q_{2r}(x_0)\right)&=\mathrm{cap}_{p_0}\left(t^{-1}\,\left(\overline{Q_r(x_0)}\setminus(t\,\Omega)\right);t^{-1}\,Q_{2r}(x_0)\right)\\
&=\mathrm{cap}_{p_0}\left(\overline{Q_{r/t}(x_0/t)}\setminus\Omega;Q_{2r/t}(x_0/t)\right).
\end{split}
\]
By using \eqref{flup!} to estimate the right-hand side, we conclude. 
\end{proof}

The following technical lemma will be needed in order to control Lebesgue norms on certain cubes lying ``sufficiently close'' to the boundary. This is a slight generalization of \cite[Lemma 3.8]{Par}, with cubes in place of balls. It will be needed in order to properly fix the argument used in\footnote{In \cite{KK}, the proof of the main result is slightly incorrect, since it uses the thickness property of $\mathbb{R}^N\setminus\Omega$ also in the case of balls lying partly outside $\Omega$, but whose center still belongs to $\Omega$. The paper by Parviainen \cite{Par} fixed this issue in order to prove the parabolic counterpart of \cite[Theorem 1.1]{KK}.} \cite{KK}.
\begin{Lemma} 
\label{lm:capacity_cubes}
Let $1\le p_0<\infty$ and let $\Omega \subsetneq  \mathbb{R}^N$ be an open set such that $\Omega\in \mathrm{UT}_{p_0}(c_0,r_0)$.
We fix a point $x_0 \in \Omega$ such that
\[
\mathfrak{d}(x_0):=\inf_{y\in\partial\Omega} \|x_0-y\|_{\ell^\infty}.
\]
There exists a constant $C_2 >1$, depending only on $N, p_0$, and $c_0$, such that
  \[
    \begin{split}
        \mathrm{cap}_{p_0}\left(\overline{Q_r(x_0)} \setminus \Omega; Q_{2r}(x_0)\right) &\geq \frac{1}{C_2}\,\left(\frac{r-\mathfrak{d}(x_0)}{2\,r+\mathfrak{d}(x_0)}\right)^{p_0}\\
        &\times(r - \mathfrak{d}(x_0))^{N-p_0},\qquad \text{for every}\ \mathfrak{d}(x_0)<r\le r_0+\mathfrak{d}(x_0).
        \end{split}
\]
\end{Lemma}
\begin{proof}
Since $\Omega$ is a proper open subset of $\mathbb{R}^N$, its boundary $\partial\Omega$ is closed and non-empty. Therefore, we can find a point $y_0 \in \partial\Omega$ such that 
\[
\|y_0 - x_0\|_\infty =\mathfrak{d}(x_0).
\] 
For simplicity, in what follows we write $\mathfrak{d}$ in place of $\mathfrak{d}(x_0)$.
For every $r>\mathfrak{d}$, we observe that 
\[
Q_{r-\mathfrak{d}}(y_0) \subseteq Q_r(x_0)\qquad \text{and}\qquad Q_{2r}(x_0) \subseteq  Q_{2r+\mathfrak{d}}(y_0).
\] 
By the monotonicity properties of the capacity, we can estimate from below
\[
\begin{split}
        \mathrm{cap}_{p_0}\left(\overline{Q_r(x_0)} \setminus \Omega; Q_{2r}(x_0)\right)& \geq \mathrm{cap}_{p_0}\left(\overline{Q_{r-\mathfrak{d}}(y_0)} \setminus \Omega; Q_{2r}(x_0)\right)\\
        & \geq \mathrm{cap}_{p_0}\left(\overline{Q_{r-\mathfrak{d}}(y_0)} \setminus \Omega; Q_{2r+\mathfrak{d}}(y_0)\right).
\end{split}
\]
We can now apply \cite[Proposition 13.1.1.1]{Maz}, in order to further estimate from below the last term as follows
\[
\mathrm{cap}_{p_0}\left(\overline{Q_{r-\mathfrak{d}}(y_0)} \setminus \Omega; Q_{2r+\mathfrak{d}}(y_0)\right)\ge c\,\left(\frac{r-\mathfrak{d}}{2\,r+\mathfrak{d}}\right)^{p_0}\,\mathrm{cap}_{p_0}\left(\overline{Q_{r-\mathfrak{d}}(y_0)} \setminus \Omega; Q_{2(r-\mathfrak{d})}(y_0)\right),
\]
where $c=c(N,p_0)>0$. 
It is now sufficient to use the property of $p_0-$thickness, by observing that $0<r-\mathfrak{d}\le r_0$, for every $r$ as in the statement.
\end{proof}
The next result will be crucial in order to achieve the desired gain of regularity. This is analogous to \cite[Lemma 3.1]{KK}, but we will avoid here to use the fine properties of Sobolev functions.
 It can be obtained by a density argument, as a straightforward consequence of a classical result by Maz'ya.
\begin{Lemma}[Maz'ya-Poincar\'e-Sobolev inequality]
\label{lm:sobolev}
Let $1\le p_0\le p<\infty$ and let $\theta\ge 1$ be such that
\[
 \left\{ \begin{array}{ll}
   \theta\le p_0^*,& \text{if}\ 1\le p_0<N,\\
   \theta<\infty,& \text{if}\ p_0\ge N.\\
   \end{array}
   \right.\qquad \text{where}\ p_0^*=\frac{N\,p_0}{N-p_0}.
\]
Let $\Omega\subsetneq\mathbb{R}^N$ be an open set such that $\Omega\in \mathrm{UT}_{p_0}(c_0,r_0)$.
 There exists a constant $C_3=C_3(N,p_0,\theta,c_0)>0$ such that for every $u\in W^{1,p}_0(\Omega)$ we have
\begin{equation*}
 \left(\fint_{Q_r(x_0)} |u|^\theta\,dx\right)^\frac{p_0}{\theta}\le C_3\,r^{p_0}\,\fint_{Q_r(x_0)} |\nabla u|^{p_0}\,dx,
\end{equation*}
where $Q_r(x_0)$ is such that $0<r\le r_0$ and
\begin{itemize}
\item either $x_0\in\mathbb{R}^N\setminus \Omega$;
\vskip.2cm
\item or $x_0\in\Omega$ with 
\begin{equation}
\label{distanzia}
\mathfrak{d}(x_0):=\min_{y\in\partial\Omega}\|x_0-y\|_{\ell^\infty}\le \frac{2}{3}\,r.
\end{equation}
\end{itemize}
\end{Lemma}
\begin{proof}
By definition of $W^{1,p}_0(\Omega)$, there exists a sequence $\{u_n\}_{n\in\mathbb{N}}\subseteq C^\infty_0(\Omega)$ such that
\begin{equation}
\label{conv_norma}
\lim_{n\to\infty} \|u_n-u\|_{W^{1,p}(\Omega)}=0.
\end{equation}
By extending the functions to be $0$ on the complement $\mathbb{R}^N\setminus\Omega$, we can thus in particular infer that
\[
\lim_{n\to\infty}\fint_{Q_r(x_0)} |\nabla u_n-\nabla u|^{p_0}\,dx=0,
\]
for $1\le p_0\le p$ and for every cube as in the statement. As for the $L^\theta$ norm of $u_n-u$, we observe that for $1\le \theta\le p$ we directly get
\begin{equation}
\label{gama}
\lim_{n\to\infty}\fint_{Q_r(x_0)} |u_n-u|^\theta\,dx=0,
\end{equation}
by \eqref{conv_norma} and Jensen's inequality. On the other hand, for an exponent $\theta>p$ satisfying the restrictions of the statement, it is sufficient to use the following standard Gagliardo-Nirenberg interpolation inequality
\[
		\|u_n-u\|_{L^\theta(\mathbb{R}^N)}\le G_{N,p,\theta}\,\|\nabla u_n-\nabla u\|_{L^p(\mathbb{R}^N)}^\vartheta\,\|u_n-u\|_{L^p(\mathbb{R}^N)}^{1-\vartheta},\qquad \text{for every}\ n\in\mathbb{N},
\]
see for example \cite[Theorem 3.8.1]{BraBook} or \cite[Theorem 12.83]{Leo}.	
Here the exponent $0<\vartheta\le 1$ is dictated by scale invariance, its precise expression does not matter. From \eqref{conv_norma}, the previous estimate implies that 
\[
\lim_{n\to\infty} \|u_n-u\|_{L^\theta(\mathbb{R}^N)}=0,
\]
and thus we get \eqref{gama}, as well.
\vskip.2cm\noindent
The previous discussion shows that it is sufficient to prove the claimed inequality for a function $u\in C^\infty_0(\Omega)$. As before, we consider $u$ to be defined on the whole $\mathbb{R}^N$, by extending it to $0$ in $\mathbb{R}^N\setminus\Omega$. In this case, we can directly appeal to \cite[Theorem 14.1.2]{Maz} and obtain
\[
\mathrm{cap}_{p_0}\left(\overline{Q_r(x_0)}\setminus\Omega;Q_{2\,r}(x_0)\right)\,\left(\int_{Q_r(x_0)} |u|^\theta\,dx\right)^\frac{p_0}{\theta}\le C\,r^{N\,\frac{p_0}{\theta}}\,\int_{Q_r(x_0)} |\nabla u|^{p_0}\,dx,
\]
for every $x_0\in \mathbb{R}^N$ and every $r>0$.
Here the constant $C=C(N,p_0,\theta)>0$. We need to estimate from below the capacity of $\overline{Q_r(x_0)}\setminus\Omega$. If $x_0\in\mathbb{R}^N\setminus\Omega$, 
we can use the $p_0-$thickness assumption on $\mathbb{R}^N\setminus\Omega$ and get for every $0<r\le r_0$
\[
\begin{split}
\mathrm{cap}_{p_0}\left(\overline{Q_r(x_0)}\setminus\Omega;Q_{2\,r}(x_0)\right)&\ge c_0\,\mathrm{cap}_{p_0}\left(\overline{Q_r(x_0)};Q_{2\,r}(x_0)\right)\\
&= c_0\,\mathrm{cap}_{p_0}\left(\overline{Q_1};Q_{2}\right)\,r^{N-p_0}.
\end{split}
\]
Observe that we used the scaling properties of the relative capacity. On the other hand, if $x_0\in\Omega$ and \eqref{distanzia} holds, 
we can apply Lemma \ref{lm:capacity_cubes} and obtain
\[
\begin{split}
   \mathrm{cap}_{p_0}\left(\overline{Q_r(x_0)} \setminus \Omega; Q_{2r}(x_0)\right) &\geq \frac{1}{C_2}\,\left(\frac{r-\mathfrak{d}(x_0)}{2\,r+\mathfrak{d}(x_0)}\right)^{p_0}\,(r - \mathfrak{d}(x_0))^{N-p_0}\\
   &\ge \frac{1}{C}\,\left(\frac{1}{8}\right)^{p_0}\,\left(\frac{1}{3}\right)^{N-p_0}\,r^{N-p_0}.
   \end{split}
\]
The desired conclusion now follows.
\end{proof}

\section{An approximating problem}
\label{sec:3}

\subsection{Construction of the approximating sequence}
As explained in the Introduction, in order to achieve the desired regularity result we will need to construct a suitable sequence $\{u_n\}_{n\in\mathbb{N}}\subseteq W^{1,p\,(1+\delta)}_0(\Omega)$ weakly converging to the solution $u$ of the Main Theorem.
\par
To this aim, it is useful to remark that $u$ is the unique minimizer of the following strictly convex functional
\[
\mathfrak{F}(\varphi)=\int_\Omega G(\nabla \varphi)\,dx-\int_\Omega f\,\varphi\,dx,\qquad \text{for every}\ \varphi\in W^{1,p}_0(\Omega).
\]
We are going to construct such a sequence through a penalized minimization problem. This is the content of the next result.
\begin{Proposition}
\label{prop:familyofproblems}
Under the assumptions of the Main Theorem, we set 
\begin{equation}
\label{qeccoti!}
q:=2\,p,
\end{equation}
and we consider the family of energy functionals
\[
\mathfrak{F}_n(\varphi)=\int_\Omega G(\nabla \varphi)\,dx+\frac{\alpha_n}{q}\,\int_\Omega |\nabla \varphi|^q\,dx-\int_\Omega f\,\varphi\,dx,\quad \text{for every}\ \varphi\in W^{1,p}_0(\Omega)\cap W^{1,q}_0(\Omega),
\]
where $\{\alpha_n\}_{n\in\mathbb{N}}$ is a sequence of positive numbers, such that
\[
\lim_{n\to\infty} \alpha_n=0.
\]
Then, for every $n\in\mathbb{N}$ there exists a unique minimizer $u_n\in  W^{1,p}_0(\Omega)\cap W^{1,q}_0(\Omega)$ for $\mathfrak{F}_n$. Moreover, $u_n$ is a weak solution of the equation
\begin{equation}
\label{EL}
-\mathrm{div\,}\nabla G(\nabla u_n)-\alpha_n\,\mathrm{div\,} (|\nabla u_n|^{q-2}\,\nabla u_n)= f,\qquad \text{in}\ \Omega.
\end{equation}
\end{Proposition}
\begin{proof}
Thanks to the choice of $\gamma_0$ and the assumption $\lambda_p(\Omega)>0$, by Proposition \ref{prop:inradius_finito}
we have that the embedding
\[
W^{1,p}_0(\Omega)\hookrightarrow L^{\gamma'_0}(\Omega),
\]
is continuous. 
Thus, for every $\varphi\in W^{1,p}_0(\Omega)$ we have 
\begin{equation}
\label{pancarre}
\left|\int_\Omega f\,\varphi\,dx\right|\le \|f\|_{L^{\gamma_0}(\Omega)}\,\|\varphi\|_{L^{\gamma_0'}(\Omega)}\le \|f\|_{L^{\gamma_0}(\Omega)}\,\frac{\|\nabla \varphi\|_{L^p(\Omega)}}{\Big(\lambda_{p,\gamma_0'}(\Omega)\Big)^\frac{1}{p}}.
\end{equation}
We recall that $\lambda_{p,\gamma_0'}(\Omega)$ is defined in \eqref{costanti}.
This shows that $\mathfrak{F}_n$ is well-defined on $W^{1,p}_0(\Omega)\cap W^{1,q}_0(\Omega)$. Moreover, since $\lambda_p(\Omega)>0$, by Proposition \ref{prop:inradius_finito} we have that
\[
\varphi\mapsto \|\nabla \varphi\|_{L^p(\Omega)},
\]
is an equivalent norm on $W^{1,p}_0(\Omega)$. The fact that $\lambda_p(\Omega)>0$ implies $\lambda_q(\Omega)>0$, as well (see for example \cite[Proposition 2.1]{BBV}). Thus, also the space $W^{1,q}_0(\Omega)$ can be equivalently normed by
\[
\varphi\mapsto \|\nabla \varphi\|_{L^q(\Omega)}.
\]
In force of all these preliminary observations, existence of a minimizer $u_n$ follows by a standard application of the Direct Methods in the Calculus of Variations. Uniqueness is a consequence of the strict convexity of the functionals. Finally, the fact that $u_n$ weakly solves \eqref{EL} simply follows from the minimality, once noticed that the functional $\mathfrak{F}_n$ is Gateaux differentiable.
\end{proof}
We now prove that the sequence of minimizers $\{u_n\}_{n\in\mathbb{N}}$ converges to the original solution $u$.
\begin{Proposition}
\label{prop:convmin}
With the notation of Proposition \ref{prop:familyofproblems}, we have
\begin{equation}
\label{convmin}
\lim_{n\to\infty} \left(\min_{\varphi\in W^{1,p}_0(\Omega)\cap W^{1,q}_0(\Omega)}\mathfrak{F}_n(\varphi)\right)=\min_{\varphi\in W^{1,p}_0(\Omega)}\mathfrak{F}(\varphi).
\end{equation}
Moreover, we have 
\[
u_n\stackrel{n\to\infty}{\rightharpoonup} u,\qquad \text{in}\ W^{1,p}(\Omega),
\]
where $u\in W_0^{1,p}(\Omega)$ is the unique minimizer of $\mathfrak{F}$. Finally, we have
\begin{equation}
\label{upp}
\limsup_{n\to\infty} \int_\Omega \Big[|\nabla u_n|^p+\alpha_n\,|\nabla u_n|^q\Big]\,dx\le C_{g_1,g_2,q}\,\int_\Omega |\nabla u|^p\,dx.
\end{equation}
\end{Proposition}
\begin{proof}
In order to prove \eqref{convmin}, observe that for every $\varphi\in C^\infty_0(\Omega)$ we have
\begin{equation*}
\begin{split}
\mathfrak{F}(\varphi)=\int_\Omega G(\nabla \varphi)\,dx-\int_\Omega f\,\varphi\,dx= \lim_{n\to\infty}\mathfrak{F}_n(\varphi)&\ge \limsup_{n\to\infty}\mathfrak{F}_n(u_n)\\
&=\limsup_{n\to\infty} \left(\min_{\psi\in W^{1,p}_0(\Omega)\cap W^{1,q}_0(\Omega)}\mathfrak{F}_n(\psi)\right).
\end{split}
\end{equation*} 
By the arbitrariness of $\varphi\in C^\infty_0(\Omega)$, this shows that
\begin{equation*}
\limsup_{n\to\infty} \left(\min_{\varphi\in W^{1,p}_0(\Omega)\cap W^{1,q}_0(\Omega)}\mathfrak{F}_n(\varphi)\right)\le \inf_{\varphi\in C^\infty_0(\Omega)}\mathfrak{F}(\varphi)=\min_{\varphi\in W^{1,p}_0(\Omega)}\mathfrak{F}(\varphi)<+\infty.
\end{equation*}
To get the reverse inequality, we observe that the previous estimate implies in particular that $\mathfrak{F}_n(u_n)\le M<+\infty$, for every $n\in\mathbb{N}$. Thus, by recalling that $\alpha_n>0$ we get 
\begin{equation*}
\begin{split}
M&\ge \int_\Omega G(\nabla u_n)\,dx-\int_\Omega f\,u_n\,dx\\
&\ge g_1\,\int_\Omega |\nabla u_n|^p\,dx-\|f\|_{L^{\gamma_0}(\Omega)}\,\frac{\|\nabla u_n\|_{L^p(\Omega)}}{\Big(\lambda_{p,\gamma_0'}(\Omega)\Big)^\frac{1}{p}}.
\end{split}
\end{equation*}
In the last inequality, we used \eqref{pancarre}.
By using Young's inequality, for every $\varepsilon>0$ we get
\[
M\ge g_1\,\int_\Omega |\nabla u_n|^p\,dx-\frac{1}{\Big(\lambda_{p,\gamma_0'}(\Omega)\Big)^\frac{1}{p-1}}\,\frac{(p-1)\,\varepsilon^{-\frac{1}{p-1}}}{p}\,\|f\|_{L^{\gamma_0}(\Omega)}^{p'}-\frac{\varepsilon}{p}\,\int_\Omega |\nabla u_n|^p\,dx.
\]
It is now sufficient to choose $\varepsilon=(g_1\,p)/2$, so to finally get
\[
\int_\Omega |\nabla u_n|^p\,dx\le C=C(N,p,\lambda_{p,\gamma_0'}(\Omega),g_1,\|f\|_{L^{\gamma_0}(\Omega)}),\qquad \text{for every}\ n\in\mathbb{N}.   
\]
This proves that $\{u_n\}_{n\in\mathbb{N}}$ is a bounded sequence in $W^{1,p}_0(\Omega)$. Thus, up to a subsequence, we have that $u_n$ weakly converges to $v$ in $W^{1,p}(\Omega)$. Moreover, we still have $v\in W^{1,p}_0(\Omega)$, since the latter is weakly closed, as well.
\par
In light of all these informations, we get
\begin{equation*}
\begin{split}
\min_{\varphi\in W^{1,p}_0(\Omega)}\mathfrak{F}(\varphi)\le \mathfrak{F}(v)\le \liminf_{n\to\infty}\mathfrak{F}(u_n)&\le \liminf_{n\to\infty}\mathfrak{F}_n(u_n)\\
&=\liminf_{n\to\infty} \left(\min_{\varphi\in W^{1,p}_0(\Omega)\cap W^{1,q}_0(\Omega)}\mathfrak{F}_n(\varphi)\right).
\end{split}
\end{equation*}
In the second inequality, we used the lower semicontinuity of the functional $\mathfrak{F}$, 
with respect to the weak convergence of the gradients. This eventually establishes \eqref{convmin}. 
\vskip.2cm\noindent
We can now ensure that all the inequalities in the previous estimates must hold as equalities. Thus, the limit function $v=u$ is the unique minimizer of $\mathfrak{F}$. This also implies that
\[
\lim_{n\to\infty}\left[\int_\Omega G(\nabla u_n)\,dx+\frac{\alpha_n}{q}\,\int_\Omega |\nabla u_n|^q\,dx-\int_\Omega f\,u_n\,dx\right]=\int_\Omega G(\nabla u)\,dx-\int_\Omega f\,u\,dx.
\]
By using the continuous embedding $W^{1,p}_0(\Omega)\hookrightarrow L^{\gamma_0'}(\Omega)$, we have that $\{u_n\}_{n\in\mathbb{N}}$ weakly converges\footnote{To be more precise, for $p>N$ and $\gamma_0'=\infty$, the sequence $\ast-$weakly converges.} also in $L^{\gamma_0'}(\Omega)$. Thus, we get 
\[
\lim_{n\to\infty}\int_\Omega f\,u_n\,dx=\int_\Omega f\,u\,dx.
\] 
By comparing the last two equations, we then obtain
\[
\lim_{n\to\infty}\left[\int_\Omega G(\nabla u_n)\,dx+\frac{\alpha_n}{q}\,\int_\Omega |\nabla u_n|^q\,dx\right]=\int_\Omega G(\nabla u)\,dx.
\]
In particular, by recalling the assumption on $G$, we get
\[
\limsup_{n\to\infty} \int_\Omega \Big[|\nabla u_n|^p+\alpha_n\,|\nabla u_n|^q\Big]\,dx\le \frac{g_2\,q}{\min\{g_1\,q,1\}}\,\int_\Omega |\nabla u|^p\,dx.
\]
The proof is now concluded. 
\end{proof}
\subsection{Local energy estimates}
By using the equation \eqref{EL}, we can in particular obtain some suitable Caccioppoli inequalities for the minimizers $u_n$. 
\begin{Lemma}[Caccioppoli inequality -- internal cubes]
\label{cacciopp}
\label{lm:caccioppoli}
Under the assumption of Proposition \ref{prop:familyofproblems}, let $Q_r(x_0)\subseteq \Omega$. Then, for every $0<\sigma<1$ and every $\mathfrak{m}\in\mathbb{R}$ we have
\[
\begin{split}
\fint_{Q_{\sigma r}(x_0)}\Big[|\nabla u_n|^p+\alpha_n\,|\nabla u_n|^q\Big]\,dx &\le \frac{C_4}{\sigma}\,\fint_{Q_r(x_0)} \left[\frac{|u_n-\mathfrak{m}|^p}{(1-\sigma)^p\,r^p}+\alpha_n\,\frac{|u_n-\mathfrak{m}|^q}{(1-\sigma)^q\,r^q}\right]\,dx\\
&+ \frac{C_4}{\sigma}\,r^{p'}\left(\fint_{Q_r(x_0)}|f|^{\gamma_0}\,dx\right)^\frac{p'}{\gamma_0}.
\end{split}
\]
The constant $C$ depends only on the dimension $N$, the exponents $p,q,\gamma_0$ and the structural constants $g_1,g_2$, but not on $n$.
\end{Lemma}
\begin{proof}
Let us fix $0<\sigma<1$ and take a Lipschitz cut-off function $\eta$ such that 
\[
0\le \eta\le 1,\qquad \eta=1\ \text{on}\ Q_{\sigma r}(x_0),\qquad \eta=0 \ \text{on}\ \partial Q_{r}(x_0),\qquad |\nabla\eta|\le \frac{1}{(1-\sigma)\,r}.
\] 
We then test the weak formulation of \eqref{EL} with $\varphi=\eta^q\,(u_n-\mathfrak{m})\in W^{1,p}_0(\Omega)\cap W_0^{1,q}(\Omega)$. This yields
\[
\begin{split}
\int_{Q_r(x_0)}\Big[\langle\nabla G(\nabla u_n),\nabla u_n\rangle+\alpha_n\,|\nabla u_n|^q\Big]\,\eta^q\,dx&+q\,\int_{Q_r(x_0)}\langle \nabla G(\nabla u_n),\nabla \eta\rangle \,\eta^{q-1}\,(u_n-\mathfrak{m})\,dx\\
&+q\,\alpha_n\,\int_{Q_r(x_0)}\langle |\nabla u_n|^{q-2}\,\nabla u_n,\nabla \eta\rangle \,\eta^{q-1}\,(u_n-\mathfrak{m})\,dx\\
&=\int_{Q_r(x_0)} f\,\eta^q\,(u_n-\mathfrak{m})\,dx.
\end{split}
\]
By using the assumptions on $G$, Lemma \ref{lm:derivataG} and Cauchy-Schwarz inequality, we get in particular
\[
\begin{split}
\int_{Q_r(x_0)}\Big[g_1\,|\nabla u_n|^p+\alpha_n\,|\nabla u_n|^q\Big]\,\eta^q\,dx&\le (2^p\,g_2-g_1)\,q\,\int_{Q_r(x_0)}|\nabla u_n|^{p-1}\,|\nabla \eta| \,\eta^{q-1}\,|u_n-\mathfrak{m}|\,dx\\
&+q\,\alpha_n\,\int_{Q_r(x_0)} |\nabla u_n|^{q-1}\,|\nabla\eta| \,\eta^{q-1}\,|u_n-\mathfrak{m}|\,dx\\
&+\int_{Q_r(x_0)} f\,\eta^q\,(u_n-\mathfrak{m})\,dx.
\end{split}
\]
By using Young's inequality, we can infer that for every $\varepsilon>0$
\[
\begin{split}
\int_{Q_r(x_0)}\Big[|\nabla u_n|^p+\alpha_n\,|\nabla u_n|^q\Big]\,\eta^q\,dx&\le C\,\varepsilon\,\int_{Q_r(x_0)} |\nabla u_n|^p\,\eta^{(q-1)\,{p'}}\,dx\\
&+C\,\varepsilon^{1-p}\,\int_{Q_r(x_0)} |\nabla \eta|^p\,|u_n-\mathfrak{m}|^p\,dx\\
&+C\,\alpha_n\,\varepsilon\,\int_{Q_r(x_0)} |\nabla u_n|^{q}\,\eta^{q}\,dx\\
&+C\,\alpha_n\,\varepsilon^{1-q}\int_{Q_r(x_0)}|\nabla\eta|^q\,|u_n-\mathfrak{m}|^q\,dx\\
&+C\,\int_{Q_r(x_0)} f\,\eta^q\,(u_n-\mathfrak{m})\,dx,
\end{split}
\]
for a constant $C=C(p,q,g_1,g_2)>0$.
Observe that, since $0\le \eta\le 1$ and $p<q$, we have 
\[
\eta^{(q-1)\,{p'}}\le \eta^q.
\]
By using this fact, choosing $\varepsilon=1/(2\,C)$
and then reabsorbing the two terms containing $\nabla u_n$ in the right-hand side, we get after some simple manipulations
\begin{equation}
\label{precaccio}
\begin{split}
\int_{Q_r(x_0)}\Big[g_1\,|\nabla u_n|^p+\alpha_n\,|\nabla u_n|^q\Big]\,\eta^q\,dx&\le C\,\int_{Q_r(x_0)} |\nabla \eta|^p\,|u_n-\mathfrak{m}|^p\,dx\\
&+C\,\alpha_n\,\int_{Q_r(x_0)}|\nabla\eta|^q\,|u_n-\mathfrak{m}|^q\,dx\\
&+C\,\int_{Q_r(x_0)} f\,\eta^q\,(u_n-\mathfrak{m})\,dx,
\end{split}
\end{equation}
possibly for a different constant $C=C(p,q,g_1,g_2)>0$.
We now estimate the last term, thanks to H\"older and Poincar\'e-Sobolev inequalities. Indeed, thanks to the choice of the exponent $\gamma_0$, we have
\[
\left\{\begin{array}{rl}
\gamma_0'< p^*,& \text{if}\ 1<p<N,\\
1<\gamma_0'<\infty,& \text{if}\ p=N,\\
\gamma_0'\le \infty,& \text{if}\ p>N.
\end{array}
\right.
\]
Thus, we obtain
\[
\begin{split}
\int_{Q_r(x_0)} f\,\eta^q\,(u_n-\mathfrak{m})\,dx&\le \left\|f\,\eta^{q-\frac{q}{p}}\right\|_{L^{\gamma_0}(Q_r(x_0))}\,\left\|\eta^\frac{q}{p}\,(u_n-\mathfrak{m})\right\|_{L^{\gamma_0'}(Q_r(x_0))}\\
&\le \frac{\left \|f\,\eta^{q-\frac{q}{p}}\right\|_{L^{\gamma_0}(Q_r(x_0))}}{\big(\lambda_{p,\gamma_0'}(Q_r(x_0))\big)^\frac{1}{p}}\,\left(\int_{Q_r(x_0)}\left|\eta^\frac{q}{p}\,\nabla u_n+\frac{q}{p}\,(u_n-\mathfrak{m})\,\eta^{\frac{q}{p}-1}\,\nabla \eta\right|^p\,dx\right)^\frac{1}{p}\\
&\le \frac{r^{\frac{N}{\gamma_0'}+1-\frac{N}{p}}}{\big(\lambda_{p,\gamma_0'}(Q_1(x_0))\big)^\frac{1}{p}}\,\left\|f\,\eta^{q-\frac{q}{p}}\right\|_{L^{\gamma_0}(Q_r(x_0))}\,\left(\int_{Q_r(x_0)}|\nabla u_n|^p\,\eta^q\,dx\right)^\frac{1}{p}\\
&+\frac{q}{p}\,\frac{r^{\frac{N}{\gamma_0'}+1-\frac{N}{p}}}{\big(\lambda_{p,\gamma_0'}(Q_1(x_0))\big)^\frac{1}{p}}\,\left\|f\,\eta^{q-\frac{q}{p}}\right\|_{L^{\gamma_0}(Q_r(x_0))}\\
&\times\left(\int_{Q_r(x_0)}|u_n-\mathfrak{m}|^p\,\eta^{q-p}\,|\nabla \eta|^p\,dx\right)^\frac{1}{p}.
\end{split}
\]
Observe that we used the scaling properties of the Sobolev-Poincar\'e constant $\lambda_{p,\gamma_0'}$, to infer that
\[
\lambda_{p,\gamma_0'}(Q_r(x_0))=r^{N-p-\frac{p}{\gamma_0'}\,N}\,\lambda_{p,\gamma_0'}(Q_1(x_0)).
\]
By using Young's inequality on both terms, we obtain for every $\varepsilon>0$
\[
\begin{split}
\int_{Q_r(x_0)} f\,\eta^q\,(u_n-\mathfrak{m})\,dx&\le \left(\frac{r^{\frac{N}{\gamma_0'}+1-\frac{N}{p}}}{\big(\lambda_{p,\gamma_0'}(Q_1(x_0))\big)^\frac{1}{p}}\right)^{p'}\,\frac{p-1}{p}\,\varepsilon^{-\frac{1}{p-1}}\,\left\|f\,\eta^{q-\frac{q}{p}}\right\|^{p'}_{L^{\gamma_0}(Q_r(x_0))}\\
&+\frac{\varepsilon}{p}\,\int_{Q_r(x_0)}|\nabla u_n|^p\,\eta^q\,dx\\
&+\left(\frac{r^{\frac{N}{\gamma_0'}+1-\frac{N}{p}}}{\big(\lambda_{p,\gamma_0'}(Q_1(x_0))\big)^\frac{1}{p}}\right)^{p'}\,\frac{p-1}{p}\,\left(\frac{q}{p}\right)^{p'}\,\left\|f\,\eta^{q-\frac{q}{p}}\right\|^{p'}_{L^{\gamma_0}(Q_r(x_0))}\\
&+\frac{1}{p}\,\int_{Q_r(x_0)}|u_n-\mathfrak{m}|^p\,|\nabla \eta|^p\,dx.
\end{split}
\] 
By using this estimate in \eqref{precaccio} and choosing $\varepsilon=p/4$, we can again absorb the term containing $\nabla u_n$ and obtain
\[
\begin{split}
\int_{Q_r(x_0)}\Big[|\nabla u_n|^p+\alpha_n\,|\nabla u_n|^q\Big]\,\eta^q\,dx&\le C\,\int_{Q_r(x_0)} |\nabla \eta|^p\,|u_n-\mathfrak{m}|^p\,dx\\
&+C\,\alpha_n\,\int_{Q_r(x_0)}|\nabla\eta|^q\,|u_n-\mathfrak{m}|^q\,dx\\
&+C\,\left(r^{\frac{N}{\gamma_0'}+1-\frac{N}{p}}\,\left\|f\,\eta^{q-\frac{q}{p}}\right\|_{L^{\gamma_0}(Q_r(x_0))}\right)^{p'},
\end{split}
\]
for a constant $C=C(N,p,q,g_1,g_2,\gamma_0)>0$. By using the properties of $\eta$, we can finally eliminate it and obtain the claimed estimate.
\end{proof}
In order to get estimates near the boundary, we will also need the following variant of the previous result. This time, we do not require the cubes to be contained in $\Omega$.
\begin{Lemma}[Caccioppoli inequality -- general cubes]
\label{lm:caccioppoli_ext}
Let $Q_r(x_0)$ be such that $Q_r(x_0)\cap \Omega\not=\emptyset$. Then, for every $0<\sigma<1$ we have
\[
\begin{split}
\fint_{Q_{\sigma r}(x_0)}\Big[|\nabla u_n|^p+\alpha_n\,|\nabla u_n|^q\Big]\,dx &\le \frac{C_4}{\sigma}\,\fint_{Q_r(x_0)} \left[\frac{|u_n|^p}{(1-\sigma)^p\,r^p}+\alpha_n\,\frac{|u_n|^q}{(1-\sigma)^q\,r^q}\right]\,dx\\
&+ \frac{C_4}{\sigma}\,r^{p'}\left(\fint_{Q_r(x_0)}|f|^{\gamma_0}\,dx\right)^\frac{p'}{\gamma_0},
\end{split}
\]
with the same constant $C_4$ of Lemma \ref{lm:caccioppoli}.
Here the functions $u_n$ and $f$ are extended by $0$ in the complement $\mathbb{R}^N\setminus\Omega$.
\end{Lemma}
\begin{proof}
As above, we fix $0<\sigma<1$ and take a Lipschitz cut-off function $\eta$ such that 
\[
0\le \eta\le 1,\qquad \eta=1\ \text{on}\ Q_{\sigma r}(x_0),\qquad \eta=0 \ \text{on}\ \partial Q_{r}(x_0),\qquad |\nabla\eta|\le \frac{1}{(1-\sigma)\,r}.
\] 
This time, we test the weak formulation of \eqref{EL} with $\varphi=\eta^q\,u_n\in W^{1,p}_0(\Omega)\cap W_0^{1,q}(\Omega)$. We can then repeat verbatim the estimates of Lemma \ref{lm:caccioppoli} and get the result.
\end{proof}

\section{Proof of the Main Theorem}
\label{sec:4}
We will divide the proof in two parts: in the first part, we will prove the claimed result under the further assumption
\begin{equation}
\label{scala1}
\|f\|_{L^{\gamma_0}(\Omega)}=1.
\end{equation}
In the second one, we will show how to remove this assumption by means of a scaling argument and get the desired a priori estimate.

\subsection{Part 1: estimate at scale $1$}
We thus assume \eqref{scala1} and, for ease of readability, we further divide this part in various steps.
\vskip.2cm\noindent  
{\it Step 1: set-up}. For every $n\in\mathbb{N}$, we take $u_n$ the minimizer constructed in Proposition \ref{prop:familyofproblems}.
From now on, we set for brevity
\begin{equation*}
    H_n(t):=|t|^p+\alpha_n\,|t|^q,\qquad \text{for every}\ t\in\mathbb{R}.
\end{equation*}
We take a cube $Q_{1}(x_0)$ such that $Q_{1}(x_0)\cap\Omega\not=\emptyset$. 
We wish to prove that for every $Q_r(y_0)\Subset  Q_{1}(x_0)$ we have
\begin{equation}
\label{reverseholde}
\fint_{Q_{r/2}(y_0)} H_n(|\nabla u_n|)\,dx\le C\,\left[\left(\fint_{Q_r(y_0)}H_n(|\nabla u_n|)^{1-\varepsilon}\,dx\right)^{\frac{1}{1-\varepsilon}}+\fint_{Q_r(y_0)}|f|^{\gamma_0}\,dx\right],
\end{equation}
for an exponent $\varepsilon=\varepsilon(N,p_0,p)\in(0,1)$ and a constant $C=C(N,p,q,p_0,c_0)>0$, both independent of $n$. Here, the exponent $\gamma_0$ is the same of the Main Theorem. 
\par
In order to get \eqref{reverseholde}, we must distinguish between different cases, depending on the position of the cube, or its center, with respect to $\Omega$.
\vskip.2cm\noindent
{\it Step 2: proof of \eqref{reverseholde} for internal cubes}. Let us suppose that 
\[
Q_r(y_0)\subseteq \Omega.
\]
For every $n\in\mathbb{N}$, we apply the Caccioppoli inequality of Lemma \ref{cacciopp}, with the choices
\[
\mathfrak{m}=\mathfrak{m}_n:=\overline{u_n}_{Q_r(y_0)},
\] 
and $\sigma=1/2$. This gives
\begin{equation}
\label{cacioH}
\fint_{Q_{r/2}(y_0)}H_n(|\nabla u_n|)\,dx \le C\,\fint_{Q_r(y_0)} H_n\left(\frac{u_n-\mathfrak{m}_n}{r}\right)\,dx+ C\,r^{p'}\,\left(\fint_{Q_r(y_0)}|f|^{\gamma_0}\,dx\right)^\frac{p'}{\gamma_0}.
\end{equation}
We set 
\begin{equation}
\label{exponenti}
p_1=\max\left\{\frac{Np}{N+p},1\right\}\qquad \text{and}\qquad q_1=\max\left\{\frac{Nq}{N+q},1\right\}.
\end{equation}
Observe that we have 
\[
p_1<p\le p_1^*,\qquad q_1<q\le q _1^*\qquad\text{and}\qquad \frac{q_1}{q}<\frac{p_1}{p},
\] 
hence by 
Sobolev-Poincar\'e inequality on cubes (see for example \cite[Theorem 3.17]{Gi}) we can infer 
\[
\fint_{Q_r(y_0)} \frac{|u_n-\mathfrak{m}_n|^p}{r^p}\le \frac{1}{\mu_{p_1,p}}\,\left(\fint_{Q_r(y_0} |\nabla u_n|^{p_1}\,dx\right)^{\frac{p}{p_1}},
\]
and
\[
\fint_{Q_r(y_0)} \frac{|u_n-\mathfrak{m}_n|^q}{r^q}\le \frac{1}{\mu_{q_1,q}}\,\left(\fint_{Q_r(y_0)} |\nabla u_n|^{q_1}\,dx\right)^{\frac{q}{q_1}}.
\]
By recalling the definition of $H_n$, these estimates entail that
\begin{equation}
\label{preH}
\begin{split}
      \fint_{Q_r(y_0)} H_n\left(\frac{u_n-\mathfrak{m}_n}{r}\right)\,dx&\le \frac{1}{\mu_{p_1,p}}\,\left(\fint_{Q_r(y_0)} |\nabla u_n|^{p_1}\,dx\right)^{\frac{p}{p_1}}\\
      &+ \frac{\alpha_n}{\mu_{q_1,q}}\,\left(\fint_{Q_r(y_0)}|\nabla u_n|^{q_1}\,dx\right)^{\frac{q}{q_1}}.
    \end{split}
\end{equation}
By using H\"older's inequality, the second integral can be estimated as follows
\[
\left(\fint_{Q_r(y_0)}|\nabla u_n|^{q_1}\,dx\right)^{\frac{q}{q_1}}=\left(\fint_{Q_r(y_0)}\left(|\nabla u_n|^{q}\right)^\frac{q_1}{q}\,dx\right)^{\frac{q}{q_1}}\le \left(\fint_{Q_r(y_0)}\left(|\nabla u_n|^{q}\right)^\frac{p_1}{p}\,dx\right)^{\frac{p}{p_1}},
\]
thanks to the fact that $p_1/p>q_1/q$. By using this fact in \eqref{preH}, we get
\[
\begin{split}
\fint_{Q_r(y_0)} H_n\left(\frac{u_n-\mathfrak{m}_n}{r}\right)\,dx&\le \frac{1}{\mu_{p_1,p}}\,\left(\fint_{Q_r(y_0)} \left(|\nabla u_n|^{p}\right)^\frac{p_1}{p}\,dx\right)^{\frac{p}{p_1}}\\
&+ \frac{1}{\mu_{q_1,q}}\,\left(\fint_{Q_r(y_0)}\left(\alpha_n\, |\nabla u_n|^{q}\right)^\frac{p_1}{p}\,dx\right)^{\frac{p}{p_1}}.
\end{split}
\]
We can apply the following two elementary inequalities (here $a,b\ge 0$)
\[
a^\frac{p}{p_1}+b^\frac{p}{p_1}\le (a+b)^\frac{p}{p_1}\qquad \text{and}\qquad a^\frac{p_1}{p}+b^\frac{p_1}{p}\le 2^\frac{p-p_1}{p}\,(a+b)^\frac{p_1}{p},
\]
so to obtain
\begin{equation}
\label{questoqua}
\begin{split}
\fint_{Q_r(y_0)} H_n\left(\frac{u_n-\mathfrak{m}_n}{r}\right)\,dx&\le 2^\frac{p-p_1}{p_1}\,\max\left\{\frac{1}{\mu_{p_1,p}},\frac{1}{\mu_{q_1,q}}\right\}\\
&\times\left(\fint_{Q_r(y_0)} \Big[|\nabla u_n|^{p}+\alpha_n\, |\nabla u_n|^{q}\Big]^\frac{p_1}{p}\,dx\right)^{\frac{p}{p_1}}.
\end{split}
\end{equation}
By using the definition of $H_n$ and inserting the previous estimate in \eqref{cacioH}, we obtain
\[
\fint_{Q_{r/2}(y_0)}H_n(|\nabla u_n|)\,dx \le C\,\left(\fint_{Q_r(y_0)} H_n(|\nabla u_n|)^\frac{p_1}{p}\,dx\right)^\frac{p}{p_1}+ C\,r^{p'}\left(\fint_{Q_r(y_0)}|f|^{\gamma_0}\,dx\right)^\frac{p'}{\gamma_0}.
\]
The last integral can be estimated by
\[
\begin{split}
r^{p'}\,\left(\fint_{Q_r(y_0)}|f|^{\gamma_0}\,dx\right)^\frac{p'}{\gamma_0}&=\frac{r^{p'}}{|Q_r(y_0)|^{{\frac{p'}{\gamma_0}-1}}}\,\left(\int_{Q_r(y_0)}|f|^{\gamma_0}\,dx\right)^{\frac{p'}{\gamma_0}-1}\,\fint_{Q_r(y_0)}|f|^{\gamma_0}\,dx\\
&\le\frac{r^{p'}}{|Q_r(y_0)|^{{\frac{p'}{\gamma_0}-1}}}\,\fint_{Q_r(y_0)}|f|^{\gamma_0}\,dx.
\end{split}
\]
Observe that we used \eqref{scala1}, so that
\[
\int_{Q_r(y_0)}|f|^{\gamma_0}\,dx\le \int_{\Omega}|f|^{\gamma_0}\,dx=1,
\]
and the fact that $\gamma_0\le p'$, by Assumption \ref{ass:1}. Moreover, 
by recalling that $Q_r(y_0)\subseteq Q_{1}(x_0)$, we have $r\le 1$ and
\[
\frac{r^{p'}}{|Q_r(y_0)|^{{\frac{p'}{\gamma_0}-1}}}=2^{N-\frac{N}{\gamma_0}\,p'}\,r^{p'+N-\frac{N}{\gamma_0}\,p'}\le 2^{N-\frac{N}{\gamma_0}\,p'},
\]
thanks to the fact that the exponent of $r$ is non-negative.
We thus get \eqref{reverseholde}, with
\[
\varepsilon_1=1-\frac{p_1}{p}.
\]
{\it Step 3: proof of \eqref{reverseholde} for ``peripherical'' cubes}. Let us suppose that 
\[
Q_r(y_0)\setminus\Omega\not=\emptyset\qquad \text{and}\qquad y_0\in\left(\mathbb{R}^N\setminus\Omega\right)\cup \left\{x\in\Omega\, :\, \min_{y\in\partial\Omega}\|x-y\|_{\ell^\infty}\le \frac{2}{3}\,r\right\}.
\]
This means that the center $y_0$ of the cube falls either outside $\Omega$ or ``not too far'' from the boundary $\partial\Omega$. For this cube, we will apply the Caccioppoli inequality of Lemma \ref{lm:caccioppoli_ext}, again with $\sigma=1/2$. We get
\begin{equation}
\label{cacioHfuori}
\fint_{Q_{r/2}(y_0)}H_n(|\nabla u_n|)\,dx \le C\,\fint_{Q_r(y_0)} H_n\left(\frac{u_n}{r}\right)\,dx+ C\,r^{p'}\left(\fint_{Q_r(x_0)}|f|^{\gamma_0}\,dx\right)^\frac{p'}{\gamma_0}.
\end{equation}
We proceed similarly as before, but we now have to use the vanishing of $u_n$ in $\mathbb{R}^N\setminus\Omega$, in order to get the desired reverse estimate. More precisely, we set 
\[
p_2=\max\{p_1,p_0\}\qquad \text{and}\qquad q_2=\max\{q_1,p_0\},
\]
where $p_1$ and $q_1$ are the same as in \eqref{exponenti}. By definition, we still have 
\[
p_2<p\left\{ \begin{array}{ll}
   \le p_2^*,& \text{if}\ 1\le p_0<N,\\
   <\infty,& \text{if}\ p_0\ge N.\\
   \end{array}
   \right.,\qquad q_2<q\left\{ \begin{array}{ll}
   \le q_2^*,& \text{if}\ 1\le p_0<N,\\
   <\infty,& \text{if}\ p_0\ge N.\\
   \end{array}
   \right.\qquad\text{and}\qquad \frac{q_2}{q}<\frac{p_2}{p}.
\] 
Observe in particular that $p_2\ge p_0$ and $q_2\ge p_0$, thus from the assumption $\Omega\in \mathrm{UT}_{p_0}(c_0)$ and Lemma \ref{lm:scendicapacity} we get that
\[
\Omega\in \mathrm{UT}_{p_2}(c_1)\cap \mathrm{UT}_{q_2}(c_2),
\]
as well, for two constants $c_1=c_1(N,p_0,p)$ and $c_2=c_2(N,p_0,q)$.
By using these thickness conditions and Lemma \ref{lm:sobolev}, we get
\[
\fint_{Q_r(y_0)} \frac{|u_n|^p}{r^p}\,dx\le C'_3\,\left(\fint_{Q_r(y_0)} |\nabla u_n|^{p_2}\,dx\right)^\frac{p}{p_2},
\]
and
\[
\fint_{Q_r(y_0)} \frac{|u_n|^q}{r^q}\le C''_3\,\left(\fint_{Q_r(y_0)} |\nabla u_n|^{q_2}\,dx\right)^{\frac{q}{q_2}},
\]
for two constants $C_3'=C_3'(N,p_0,p,c_0)>0$ and $C_3''=C_3''(N,p_0,q,c_0)>0$. We can now repeat the arguments used in {\it Step 2} in order to get \eqref{questoqua}. This time we get
\[
\begin{split}
\fint_{Q_r(y_0)} H_n\left(\frac{u_n}{r}\right)\,dx&\le 2^\frac{p-p_1}{p_1}\,\max\left\{C_3',C_3''\right\}\left(\fint_{Q_r(y_0)} H_n(|\nabla u_n|)^\frac{p_2}{p}\,dx\right)^{\frac{p}{p_2}}.
\end{split}
\]
We insert this estimate in \eqref{cacioHfuori} and proceed exactly as in {\it Step 2}. We get again \eqref{reverseholde}, this time with
\[
\varepsilon_2=1-\frac{p_2}{p}.
\]
{\it Step 4: proof of \eqref{reverseholde} for ``suburban'' cubes}. Finally, we consider the case 
\[
Q_r(y_0)\setminus\Omega\not=\emptyset\qquad \text{and}\qquad y_0\in\left\{x\in\Omega\, :\, \min_{y\in\partial\Omega}\|x-y\|_{\ell^\infty}>\frac{2}{3}\,r\right\}.
\]
In this case, observe that we have $Q_{(2r)/3}(y_0)\subseteq \Omega$. Thus, we can appeal again to the Caccioppoli inequality for internal cubes, i.e. Lemma \ref{cacciopp}, as in {\it Step 2}. We use it on the cube $Q_{(2r)/3}(y_0)$, by choosing $\sigma=3/4$, so that
\[
Q_{\sigma\,\frac{2r}{3}}(y_0)=Q_{\frac{r}{2}}(y_0).
\]
This gives 
\begin{equation}
\label{cacioHbis}
\fint_{Q_{\frac{r}{2}}(y_0)}H_n(|\nabla u_n|)\,dx \le C\,\fint_{Q_{\frac{2r}{3}}(y_0)} H_n\left(\frac{u_n-\mathfrak{m}_n}{r}\right)\,dx+ C\,r^{p'}\left(\fint_{Q_{\frac{2r}{3}}(y_0)}|f|^{\gamma_0}\,dx\right)^\frac{p'}{\gamma_0},
\end{equation}
where again $\mathfrak{m}_n:=\overline{u_n}_{Q_r(y_0)}$. It is easily seen that
\[
\fint_{Q_{\frac{2r}{3}}(y_0)} H_n\left(\frac{u_n-\mathfrak{m}_n}{r}\right)\,dx\le \left(\frac{3}{2}\right)^N\,\fint_{Q_{r}(y_0)} H_n\left(\frac{u_n-\mathfrak{m}_n}{r}\right)\,dx,
\]
and
\[
\fint_{Q_{\frac{2r}{3}}(y_0)}|f|^{\gamma_0}\,dx\le \left(\frac{3}{2}\right)^N\,\fint_{Q_{r}(y_0)}|f|^{\gamma_0}\,dx.
\]
Thus, from \eqref{cacioHbis} we can obtain \eqref{cacioH}, possibly with a different constant $C$ (still independent of $n$). The desired conclusion now follows by proceeding exactly as in {\it Step 2}. We then get \eqref{reverseholde}, again with
\[
\varepsilon_1=1-\frac{p_1}{p}.
\]
\vskip.2cm\noindent
{\it Step 5: local gain of integrability}.
By setting 
\[
\varepsilon=\max\{\varepsilon_1,\, \varepsilon_2\}=\max\left\{1-\frac{p_1}{p},\, 1-\frac{p_2}{p}\right\},
\]
and using H\"older inequality, we thus have established
\eqref{reverseholde} for every cube $Q_r(y_0)\Subset Q_{1}(x_0)$. We can now appeal to Lemma \ref{lm:gehringG} and infer that $H_n(\nabla u_n)\in L^{1+\delta}(Q_{1/2}(x_0))$, for some $\delta=\delta(N,p_0,p,c_0,\gamma,g_1,g_2)>0$ such that
\[
\delta\le \min\left\{1,\, \frac{\gamma-\gamma_0}{\gamma_0}\right\}.
\]
Moreover, we have the estimate
\[
\fint_{Q_{1/2}(x_0)} H_n(\nabla u_n)^{1+\delta}\,dx\le C\,\left[\left(\fint_{Q_{1}(x_0)} H_n(\nabla  u_n)\,dx\right)^{1+\delta}+\fint_{Q_{1}(x_0)} |f|^{\gamma_0\,(1+\delta)}\,dx\right].
\]
The constant $C$ does not depend on $n\in\mathbb{N}$. In particular, we also get
\begin{equation}
\label{gehrfinal}
\begin{split}
\int_{Q_{1/2}(x_0)} |\nabla u_n|^{p\,(1+\delta)}\,dx&\le C\,\left[\left(\int_{Q_{1}(x_0)} H_n(|\nabla u_n|)\,dx\right)^{1+\delta}+\int_{Q_{1}(x_0)} |f|^{\gamma_0\,(1+\delta)}\,dx\right],
\end{split}
\end{equation}
possibly for a different constant $C$, still independent of $n$.
\vskip.2cm\noindent
{\it Step 6: global uniform gain}. We can now get a global uniform estimate on $\Omega$ for $\nabla u_n$, by means of a covering argument. We consider the countable tiling of $\mathbb{R}^N$ given by the cubes
\[
\Big\{Q_{1/2}(\mathbf{i})\Big\}_{\mathbf{i}\in\mathbb{Z}^N}
\] 
with sides parallel to the coordinate axis and with disjoint interiors. We set
\[
\mathbb{Z}^N_\Omega=\{\mathbf{i}\in\mathbb{Z}^N\,:\, Q_{1/2}(\mathbf{i})\cap \Omega\not=\emptyset\},
\]
so that
\[
\sum_{\mathbf{i}\in \mathbb{Z}^N_\Omega} \int_{Q_{1/2}(\mathbf{i})\cap \Omega} |\nabla u_n|^{p\,(1+\delta)}\,dx=\int_\Omega |\nabla u_n|^{p\,(1+\delta)}\,dx.
\]
We then double the side length of the cubes, by considering $\{Q_{1}(\mathbf{i})\}_{\mathbf{i}\in\mathbb{Z}^N_\Omega}$: observe that this gives a locally finite covering of $\Omega$. More precisely, for a fixed $\mathbf{i}\in\mathbb{Z}^N_\Omega$ we have
\begin{equation*}
    \#\Big\{\mathbf{j}\in\mathbb{Z}^N_\Omega\,:\, Q_{1}(\mathbf{i})\cap Q_{1}(\mathbf{j})\Big\}\le    \#\Big\{\mathbf{j}\in\mathbb{Z}^N\,:\, Q_{1}(\mathbf{i})\cap Q_{1}(\mathbf{j})\Big\}=3^N-1.
\end{equation*}
In other words, when we double the length of the side of the tiling, we get a covering such that each cube intersects only a (universally) finite number of other cubes. This is needed in order to trigger the covering argument. In particular, we can rearrange the family $\{Q_{1}(\mathbf{i})\}_{\mathbf{i}\in\mathbb{Z}^N_\Omega}$ into $2^N$ subfamilies $\mathfrak{Q}_1,...,\mathfrak{Q}_{2^N}$ satisfying
\begin{equation*}
    Q_{1}(\mathbf{i})\cap Q_{1}(\mathbf{j})=\emptyset,\quad\text{if}\ Q_1(\mathbf{i})\not=Q_1(\mathbf{j})\in\mathfrak{Q}_k,\qquad\text{with}\,\,1\le k\le 2^N.
\end{equation*}
From this observation and estimate \eqref{gehrfinal}, we can infer   
\begin{equation}
\label{gehrfinal2}
\begin{split}  
     \int_\Omega |\nabla u_n|^{p\,(1+\delta)}\,dx&=\sum_{\mathbf{i}\in \mathbb{Z}^N_\Omega} \int_{Q_{1/2}(\mathbf{i})\cap \Omega} |\nabla u_n|^{p\,(1+\delta)}\,dx\\
        &\le C\,\sum_{\mathbf{i}\in \mathbb{Z}^N_\Omega} \left(\int_{Q_{1}(\mathbf{i})\cap\Omega} H_n(|\nabla u_n|)\,dx\right)^{1+\delta}\\
        &+C\,\sum_{\mathbf{i}\in \mathbb{Z}^N_\Omega}\int_{Q_{1}(\mathbf{i})\cap\Omega}|f|^{\gamma_0\,(1+\delta)}\,dx.
\end{split}
\end{equation}
For the last term, we simply observe that 
\[
\begin{split}
\sum_{\mathbf{i}\in \mathbb{Z}^N_\Omega}\int_{Q_{1}(\mathbf{i})\cap\Omega}|f|^{\gamma_0\,(1+\delta)}\,dx
&=\sum_{j=1}^{2^N}\left(\sum_{\{\mathbf{i}\in\mathbb{Z}^N_\Omega\, :\, Q_{1}(\mathbf{i})\in\mathfrak{Q}_j\}}\int_{Q_{1}(\mathbf{i})\cap\Omega}|f|^{\gamma_0\,(1+\delta)}\,dx\right)\\
&\le 2^N\,\int_\Omega |f|^{\gamma_0\,(1+\delta)}\,dx,
\end{split}
\]
where we used the fact that the cubes in each $\mathfrak{Q}_j$ have disjoint interiors. We can further estimate the last integral, by using \eqref{scala1} and interpolation in Lebesgue spaces: since $\gamma_0<\gamma_0\,(1+\delta)\le \gamma$, we get
\[
\int_\Omega |f|^{\gamma_0\,(1+\delta)}\,dx\le \left(\int_\Omega |f|^{\gamma}\,dx\right)^{\frac{\delta\,\gamma_0}{\gamma-\gamma_0}}.
\]
In a similar way, by using the super-additivity of convex powers, we get
\[
\begin{split}
\sum_{\mathbf{i}\in \mathbb{Z}^N_\Omega} \left(\int_{Q_{1}(\mathbf{i})\cap \Omega} H_n(|\nabla u_n|)\,dx\right)^{1+\delta}&= \sum_{j=1}^{2^N} \sum_{\{\mathbf{i}\in\mathbb{Z}^N_\Omega\, :\, Q_{1}(\mathbf{i})\in\mathfrak{Q}_j\}} \left(\int_{Q_{1}(\mathbf{i})} H_n(|\nabla u_n|)\,dx\right)^{1+\delta}\\
&\le 2^N\,\left(\int_\Omega H_n(|\nabla u_n|)\,dx\right)^{1+\delta}.
\end{split}
\]
Thus, from \eqref{gehrfinal2} we get
\begin{equation}
\label{gehrfinal3}
 \int_\Omega |\nabla u_n|^{p\,(1+\delta)}\,dx\le C\,\left(\int_\Omega H_n(|\nabla u_n|)\,dx\right)^{1+\delta}+C\,\left(\int_\Omega |f|^{\gamma}\,dx\right)^{\frac{\delta\,\gamma_0}{\gamma-\gamma_0}}.
\end{equation}
\vskip.2cm\noindent
{\it Step 7: conclusion}. We now wish to pass to the limit as $n$ goes to $\infty$.
Observe that $p<p\,(1+\delta)\le q$: the second inequality follows by recalling that $q=2\,p$ (see the definition \eqref{qeccoti!}) and using that
\[
\delta\le \min\left\{1,\, \frac{\gamma-\gamma_0}{\gamma_0}\right\}.
\]
Thus, from $u_n\in W^{1,p}_0(\Omega)\cap W^{1,q}_0(\Omega)$ we get that $u_n\in W^{1,p\,(1+\delta)}_0(\Omega)$, as well.
Moreover, by recalling \eqref{upp}, we have that the right-hand side of \eqref{gehrfinal3} is uniformly bounded in $n$. Thus, we get that 
\[
\{u_n\}_{n\in\mathbb{N}}\subseteq W_0^{1,p\,(1+\delta)}(\Omega)\qquad \text{with}\quad \sup_{n\in\mathbb{N}} \|u_n\|_{W^{1,p\,(1+\delta)}(\Omega)}<+\infty.
\]
By reflexivity, we know that $\{u_n\}_{n\in\mathbb{N}}$ weakly converges in $W^{1,p\,(1+\delta)}(\Omega)$ to some function $v$, up to a subsequence. Moreover, since $W^{1,p\,(1+\delta)}_0(\Omega)$ is weakly closed, we get that $v\in W^{1,p\,(1+\delta)}_0(\Omega)$. By uniqueness of the limit, from Proposition \ref{prop:convmin} we can thus conclude
\[
u=v\in W^{1,p\,(1+\delta)}_0(\Omega),
\]
which establishes the desired regularity result on $u$, under the condition \eqref{scala1}. 
\par
This result comes with the relevant a priori estimate: by taking the limit as $n$ goes to $\infty$ in \eqref{gehrfinal3}, using the lower semicontinuity of the left-hand side and \eqref{upp} in the right-hand side, we finally get
\begin{equation}
\label{gehrfinal4}
\int_\Omega |\nabla u|^{p\,(1+\delta)}\,dx\le C\,\left(\int_\Omega |\nabla u|^p\,dx\right)^{1+\delta}+C\,\left(\int_\Omega |f|^{\gamma}\,dx\right)^{\frac{\delta\,\gamma_0}{\gamma-\gamma_0}},
\end{equation}
possibly for a different constant $C=C(N,p,\gamma_0,\gamma,\delta,g_1,g_2)>0$.

\subsection{Part 2: estimate at a general scale} We now remove the restriction \eqref{scala1} and make the estimate scale invariant.
\vskip.2cm\noindent
{\it Step 1: vertical scalings}. 
Given the solution $u\in W^{1,p}_0(\Omega)$, we set
\[
u_\lambda(x)=\lambda\,u(x),\qquad \text{with}\ \lambda=\|f\|_{L^{\gamma_0}(\Omega)}^\frac{1}{1-p}.
\]
It is not difficult to see that $u_\lambda\in W^{1,p}_0(\Omega)$ is the weak solution of 
\[
-\mathrm{div}(\nabla G_\lambda(\nabla u_\lambda))=f_\lambda,\qquad \text{in}\ \Omega,
\]
where 
\begin{equation}
\label{Glambda}
G_\lambda(z)=\lambda^p\,G\left(\frac{z}{\lambda}\right)\qquad \text{and}\qquad f_\lambda=\lambda^{p-1}\,f.
\end{equation}
Observe that $G_\lambda$ is still $C^1$ strictly convex and satisfies
\[
g_1\,|z|^p\le G_\lambda(z)\le g_2\,|z|^p,
\]
i.e. the constants entering in the growth conditions are still the same $g_1$ and $g_2$. Moreover, we have
\[
\|f_\lambda\|_{L^{\gamma_0}(\Omega)}=1.
\]
We can thus apply the result of {\bf Part 1} to $u_\lambda$ and obtain $u_\lambda\in W^{1,p\,(1+\delta)}_0(\Omega)$, together with the estimate \eqref{gehrfinal4}.
By recalling the definition of $u_\lambda$ and $f_\lambda$, we get $u\in W^{1,p\,(1+\delta)}_0(\Omega)$ with the estimate
\[
 \int_{\Omega}|\nabla u|^{p\,(1+\delta)}\,dx\le C\,\left[\left(\int_{\Omega} |\nabla u|^p\,dx\right)^{1+\delta}+ 
\lambda^{\gamma\,(p-1)\,\frac{\delta\,\gamma_0}{\gamma-\gamma_0}-p\,(1+\delta)}\,
\left(\|f\|_{L^\gamma(\Omega)}\right)^{\frac{\delta\,\gamma_0\,\gamma}{\gamma-\gamma_0}}\right].
\]
\vskip.2cm\noindent
{\it Step 2: horizontal scalings}.  We finally adjust the a priori estimate, by making it scale invariant with respect to horizontal scalings, as well.
We fix $r>0$, given the solution $u\in W^{1,p}_0(\Omega)$, we set
\[
u_r(x)=u(r\,x).
\]
As before, we see that $u_r\in W^{1,p}_0(\Omega/r)$ is the weak solution of 
\[
-\mathrm{div}(\nabla G_r(\nabla u_r))=\widetilde{f}_r,\qquad \text{in}\ \frac{1}{r}\,\Omega,
\]
where $G_r$ is still given by \eqref{Glambda}, with $r$ in place of $\lambda$, and
\[
\widetilde{f}_r(x)=r^{p}\,f(r\,x).
\]
We notice that $\Omega/r$ is still uniformly $p_0-$thick, with the same constant $c_0$, thanks to Lemma \ref{lm:riscalothick} (used with $r_0=+\infty$). Thus, we can appeal to the previous step and get $u_r\in W^{1,p\,(1+\delta)}_0(\Omega/r)$, with the estimate
\[
\begin{split}
\left(\int_{\Omega/r} |\nabla u_r|^{p\,(1+\delta)}\,dx\right)^\frac{1}{1+\delta}&\le C\,\int_{\Omega/r} |\nabla u_r|^p\,dx\\
&+C\,\left(\|\widetilde{f}_r\|_{L^{\gamma_0}(\Omega/r)}\right)^{p'-\frac{\delta}{1+\delta}\,\frac{\gamma_0\,\gamma}{\gamma-\gamma_0}}\,\left(\|\widetilde{f}_r\|_{L^{\gamma}(\Omega/r)}\right)^{\frac{\delta}{1+\delta}\,\frac{\gamma_0\,\gamma}{\gamma-\gamma_0}}.
\end{split}
\]
By recalling the definitions of $u_r$ and $\widetilde{f}_r$, we can scale back this estimate and get the estimate 
for every $r>0$
\begin{equation}
\label{stimaprioni}
\begin{split}
\left(\frac{1}{r^{N-p\,(1+\delta)}}\,\int_\Omega |\nabla u|^{p\,(1+\delta)}\,dx\right)^\frac{1}{1+\delta}&\le \frac{C}{r^{N-p}}\,\int_\Omega |\nabla u|^p\,dx\\
&+C\,\left(\frac{1}{r^{N-p\,\gamma_0}}\,\|f\|^{\gamma_0}_{L^{\gamma_0}(\Omega)}\right)^{\frac{p'}{\gamma_0}-\frac{\delta}{1+\delta}\,\frac{\gamma}{\gamma-\gamma_0}}\\
&\times\left(\frac{1}{r^{N-p\,\gamma}}\,\|f\|^\gamma_{L^{\gamma}(\Omega)}\right)^{\frac{\delta}{1+\delta}\,\frac{\gamma_0}{\gamma-\gamma_0}}.
\end{split}
\end{equation}
We multiply both sides of \eqref{stimaprioni} by 
\[
r^{\frac{N}{1+\delta}-p},
\]
so to get (after some tedious, yet elementary, algebraic computations)
\begin{equation}
\label{stimaprioni2}
\begin{split}
\left(\int_\Omega |\nabla u|^{p\,(1+\delta)}\,dx\right)^\frac{1}{1+\delta}&\le \frac{C_1}{r^{N\frac{\delta}{1+\delta}}}\,\int_\Omega |\nabla u|^p\,dx\\
&+C_1\,r^{p'-\frac{N}{\gamma_0}\,(p'-\gamma_0)}\,\left(\|f\|_{L^{\gamma_0}(\Omega)}\right)^{p'-\frac{\delta}{1+\delta}\,\frac{\gamma\,\gamma_0}{\gamma-\gamma_0}}\,\left(\|f\|_{L^{\gamma}(\Omega)}\right)^{\frac{\delta}{1+\delta}\,\frac{\gamma\,\gamma_0}{\gamma-\gamma_0}},
\end{split}
\end{equation}
which holds for every $r>0$, where $C_1=C_1(N,p,p_0,g_1,g_2,\gamma_0,\gamma,\delta)>0$. We now want to optimize this estimate with respect to the scale parameter $r>0$.
At first, we notice that we can rewrite the second exponent on $r$ as follows
\[
\begin{split}
p'-\frac{N}{\gamma_0}\,(p'-\gamma_0)&=p'\,N\,\left(\frac{1}{N}-\frac{1}{\gamma_0}+\frac{1}{p'}\right)\\
&=p'\,N\,\left(\frac{1}{N}-\frac{1}{\gamma_0}+1-\frac{1}{p}\right)=p'\,N\,\left(\frac{1}{\gamma_0'}-\frac{N-p}{N\,p}\right).
\end{split}
\]
By recalling the assumption on $\gamma_0$, it is not difficult to see that the latter is positive.
Thus, the right-hand side of \eqref{stimaprioni2} has the form
\[
\frac{A}{r^\alpha}+B\,r^\beta,\qquad \text{with}\ A,B,\alpha,\beta>0.
\]
The previous quantity is minimal for the choice
\[
r=\left(\frac{\alpha}{\beta}\,\frac{A}{B}\right)^\frac{1}{\alpha+\beta}.
\]
By making this choice in \eqref{stimaprioni2}, we can get the desired scale invariant estimate and conclude the proof.

\section{Example: contractible sets}
\label{sec:5}

We want to complement the Main Theorem by discussing a sufficient topological condition ensuring that $\mathbb{R}^N\setminus\Omega$ is uniformly $p-$thick. Namely, we show that contractible sets satisfies this property, provided $p>N-1$. This fact should be already known, but it is not always easy to find a reference containing a proof. For this reason, for completeness we provide it: in this way, one can also get a quantitative information about the relevant constant $c_0$ of $p-$thickness. We will also show that for contractible open sets, the analytic condition {\rm (\ref{item:EOmega})} admits a simple geometric equivalent reformulation. 
\subsection{Thickness of the complement}
The first step is the following condition implying $p-$thickness of the complement. This is a consequence of \cite[Theorem 3.3]{KiKo} and Lemma \ref{equivalencecubesballs}. We give the details needed in order to keep track of the relevant constant.
\begin{Lemma}
\label{KK2}
Let $1\le p<\infty$ and let $\Omega\subsetneq \mathbb{R}^N$ be an open set. Let us suppose that there exists a constant $\mathcal{C}>0$ and a radius $r_0>$ such that
\begin{equation}\label{poincareineq}
    \int_{B_r(x_0)}|\varphi|^p\,dx\le \mathcal{C}\,r^p\,\int_{B_r(x_0)} |\nabla \varphi|^p\,dx,\qquad \text{for every}\ x_0\in\mathbb{R}^N\setminus\Omega,\ 0<r\le r_0,\ \varphi\in C^\infty_0(\Omega). 
\end{equation}
Then $\Omega\in\mathrm{UT}_{p}(\widetilde{c_0},r_0)$, with $\widetilde{c_0}$ given by
\[
\widetilde{c_0}=\frac{1}{\beta_2}\,\frac{\omega_N}{2^{p+N}}\,\min\left\{\lambda_p(B_2),\, \frac{2^p}{\mathcal{C}}\right\}\,\frac{1}{\mathrm{cap}_p\left(\overline{Q_1};Q_2\right)},
\]
and $\beta_2=\beta_2(N,p)>0$ is the same constant as in Lemma \ref{equivalencecubesballs}.
\end{Lemma}
\begin{proof}
According to \cite[Theorem 3.3]{KiKo}, the assumption on $\Omega$ is {\it equivalent} to the fact that there exists $c_0>0$ such that
\begin{equation}
\label{chicco}
\mathrm{cap}_{p}\left(\overline{B_r(x_0)}\setminus\Omega;B_{2\,r}(x_0)\right)\ge c_0\,r^{N-p},\qquad \text{for every}\ x_0\in\mathbb{R}^N\setminus\Omega,\ 0<r\le r_0.
\end{equation}
An inspection of the proof of \cite[Theorem 3.3]{KiKo} informs us that we can take
\[
c_0=\frac{\omega_N}{2^{p+N}}\,\min\left\{\lambda_p(B_2),\, \frac{2^p}{\mathcal{C}}\right\}.
\]
If we now set
\[
\widetilde{c_0}=\frac{1}{\beta_2}\,\frac{c_0}{\mathrm{cap}_p\left(\overline{Q_1};Q_2\right)},
\]
use the leftmost inequality in Lemma \ref{equivalencecubesballs} and then observe that 
\[
c_0\, r^{N-p}=\beta_2\,\widetilde{c_0}\,r^{N-p}\,\mathrm{cap}_p\left(\overline{Q_1};Q_2\right)=\beta_2\,\widetilde{c_0}\,\mathrm{cap}_p\left(\overline{Q_r(x_0)};Q_{2r}(x_0)\right),
\]
from the previous estimate \eqref{chicco} we get the desired conclusion.
\end{proof}
We now prove the claimed result for contractible sets. We start with the two-dimensional case $N=2$, which is somehow exceptional. In this case, we can even include the limit case $p=N-1=1$. This is due to the fact that  the following one-dimensional sharp Poincar\'e constant for the interval $\mathrm{I}=(0,1)$
\[
\pi_p=\inf_{\varphi\in C^\infty_0(\mathrm{I})} \left\{\|\varphi'\|_{L^p(\mathrm{I})}\, :\, \|\varphi\|_{L^p(I)}=1\right\},
\]
is positive for every $1\le p\le\infty$. We have the following
\begin{Proposition}
\label{prop:corsimply}
Let $\Omega\subsetneq \mathbb{R}^2$ be an open simply connected set. Then, for every $1\le p<\infty$ there exists a constant $c_p>0$ such that 
\[
\Omega\in\mathrm{UT}_p(c_p).
\]
\end{Proposition}
\begin{proof}
In light of Lemma \ref{KK2}, we just have to prove the validity of the boundary Poincar\'e inequality. We will use the argument by Hayman contained in \cite[Lemma 1]{Ha}. Let $\varphi\in C^\infty_0(\Omega)$ and let $x_0\in\mathbb{R}^2\setminus\Omega$. We observe that for every $\varrho>0$, we must have
\[
\partial B_\varrho(x_0)\cap(\mathbb{R}^2\setminus \Omega)\not=\emptyset.
\]
Indeed, if $B_\varrho(x_0)\subseteq\Omega$ this would violate the topological property of $\Omega$, since the loop $\gamma(t)=x_0+(\varrho\,\cos t,\varrho\,\sin t)$ could not be null-homotopic. By using polar coordinates and a slight abuse of notation, we thus get that for every $\varrho>0$ there exists $\vartheta_\varrho\in[0,2\,\pi)$ such that
\[
\varphi(\varrho,\vartheta_\varrho)=0.
\]
We can thus apply the one-dimensional Poincar\'e inequality for the function $\vartheta\mapsto\varphi(\varrho,\vartheta)$, so to get
\[
\left(\frac{\pi_p}{2\,\pi}\right)^p\,\int_0^{2\,\pi} |\varphi(\varrho,\vartheta)|^p\,d\vartheta\le \int_0^{2\,\pi} |\partial_\vartheta \varphi(\varrho,\vartheta)|^p\,d\vartheta.
\] 
We multiply both sides by $\varrho$ and then integrate over the interval $[0,r]$: this yields
\[
\begin{split}
\left(\frac{\pi_p}{2\,\pi}\right)^p\,\iint_{[0,r]\times[0,2\,\pi]} |\varphi(\varrho,\vartheta)|^p\,\varrho\,d\vartheta\,d\varrho&\le \iint_{[0,r]\times[0,2\,\pi]} |\partial_\vartheta \varphi(\varrho,\vartheta)|^p\,\varrho\,d\vartheta\,d\varrho\\
&\le \iint_{[0,r]\times[0,2\,\pi]} \left[|\partial_\varrho\varphi|^2+\frac{1}{\varrho^2}\,|\partial_\vartheta \varphi(\varrho,\vartheta)|^2\right]^\frac{p}{2}\,\varrho^{p+1}\,d\vartheta\,d\varrho\\
&\le r^p\,\iint_{[0,r]\times[0,2\,\pi]} \left[|\partial_\varrho\varphi|^2+\frac{1}{\varrho^2}\,|\partial_\vartheta \varphi(\varrho,\vartheta)|^2\right]^\frac{p}{2}\,\varrho\,d\vartheta\,d\varrho
\end{split}
\]
This proves that 
\[
\left(\frac{\pi_p}{2\,\pi}\right)^p\,\frac{1}{r^p}\,  \int_{B_r(x_0)}|\varphi|^p\,dx\le \int_{B_r(x_0)} |\nabla \varphi|^p\,dx,\qquad \text{for every}\ x_0\in\mathbb{R}^2\setminus\Omega,\ \varphi\in C^\infty_0(\Omega),
\]
as desired.
\end{proof}
The previous result can be generalized 
to contractible sets in every dimension, as announced above. However, in this case we need to have the strict inequality $p>N-1$.
\begin{Proposition}
\label{prop:contratto}
Let $N\ge 3$ and let $\Omega\subsetneq \mathbb{R}^N$ be a contractible open set. Then, for every $N-1<p<\infty$ there exists a constant $c_{N,p}>0$ such that 
\[
\Omega\in\mathrm{UT}_p(c_{N,p}).
\]
\end{Proposition}
\begin{proof}[Proof]
As before, the goal is to show the validity of the boundary Poincar\'e inequality \eqref{poincareineq}. We still use Hayman's idea, as in the two-dimensional case.
Fix $x_0\in \mathbb{R}^N\setminus\Omega$ and $\varphi\in C_0^\infty(\Omega)$. For every $\varrho>0$ we have
\[
\partial B_\varrho(x_0)\cap(\mathbb{R}^N\setminus \Omega)\not=\emptyset.
\] 
Indeed, if $\partial B_\varrho(x_0)\subseteq\Omega$ one could define the continuous map
\[
\begin{array}{rccc}
f:&\mathbb{S}^{N-1}&\to& \Omega\\
& \omega&\mapsto& x_0+\varrho\,\omega
\end{array}
\]
and this could not be null-homotopic, since $x_0\not\in\Omega$. Thus, the $(N-1)-$th homotopy group would be non-trivial, by violating the fact that $\Omega$ is contractible.
\par
The previous observation guarantees that for every $\varrho>0$ there exists $\omega_\varrho\in \mathbb{S}^{N-1}$ such that $x_0+\varrho\,\omega_\varrho\not\in\Omega$. Thus, we have $\varphi(x_0+\varrho\,\omega_\varrho)=0$ and we can apply the Poincar\'e inequality of Lemma \ref{lm:poincareS} below, to the function $\omega\mapsto \varphi(x_0+\varrho\,\omega)$. We then obtain
\[
\int_{\mathbb{S}^{N-1}}|\varphi(x_0+\varrho\,\omega)|^p\,d\sigma(\omega)\le C\,\int_{\mathbb{S}^{N-1}} |\nabla_\tau \varphi(x_0+\varrho\,\omega)|^p\,d\sigma(\omega).
\]
Here, $\nabla_\tau$ is the tangential gradient and $d\sigma$ is the standard surface measure on $\mathbb{S}^{N-1}$.
If we multiply both sides by $\varrho^{N-1}$ and integrate with respect to $\varrho\in[0,r]$, we get
\[
\begin{split}
\int_0^r\int_{\mathbb{S}^{N-1}} |\varphi(x_0+\varrho\,\omega)|^p\,\varrho^{N-1}\,d\sigma(\omega)\,d\varrho&\le C\,\int_0^r\int_{\mathbb{S}^{N-1}} |\nabla_\tau\varphi(x_0+\varrho\,\omega)|^p\,\varrho^{N-1}\,d\sigma(\omega)\,d\varrho\\
&\le C\,\int_{0}^r\int_{\mathbb{S}^{N-1}} \left[|\partial_\varrho\varphi|^2+\frac{1}{\varrho^2}\,|\nabla_\tau\varphi|^2\right]^\frac{p}{2}\,\varrho^{N+p-1}\,d\varrho\,d\sigma(\omega)\\
&\le C\,r^p\,\int_{0}^r\int_{\mathbb{S}^{N-1}} \left[|\partial_\varrho\varphi|^2+\frac{1}{\varrho^2}\,|\nabla_\tau\varphi|^2\right]^\frac{p}{2}\,\varrho^{N-1}\,d\varrho\,d\sigma(\omega).
\end{split}
\]
This finally proves that 
\[
\frac{1}{C\,r^p}\,  \int_{B_r(x_0)}|\varphi|^p\,dx\le \int_{B_r(x_0)} |\nabla \varphi|^p\,dx,\qquad \text{for every}\ x_0\in\mathbb{R}^N\setminus\Omega,\ \varphi\in C^\infty_0(\Omega),
\]
as desired.
\end{proof}
We complement this section with the following Poincar\'e inequality on the sphere, for functions vanishing at a point: it has been an essential ingredient of the previous proof.
\begin{Lemma}
\label{lm:poincareS}
Let $N\ge 3$ and let $N-1<p<\infty$. There exists a constant $C=C(N,p)>0$ such that for every $x_0\in\mathbb{S}^{N-1}$ and every $\varphi\in C^\infty(\mathbb{S}^{N-1})$ with $\varphi(x_0)=0$, we have
\[
\int_{\mathbb{S}^{N-1}}|\varphi|^p\,d\sigma\le C\,\int_{\mathbb{S}^{N-1}} |\nabla_\tau \varphi|^p\,d\sigma.
\]
\end{Lemma}
\begin{proof}
Thanks to the fact that $p>N-1$, we have the Sobolev-Morrey continuous embedding 
\[
W^{1,p}(\mathbb{S}^{N-1})\hookrightarrow C^{0,1-\frac{N-1}{p}}(\mathbb{S}^{N-1}),
\]
see \cite[Theorem 2.8]{He}. Moreover, by the same result, for every $\varphi\in C^\infty(\mathbb{S}^{N-1})$ we have
\[
[\varphi]_{C^{0,1-\frac{N-1}{p}}(\mathbb{S}^{N-1})}:=\frac{|\varphi(x)-\varphi(y)|}{d_{\mathbb{S}^{N-1}}(x,y)^{1-\frac{N-1}{p}}}\le C\,\left(\int_{\mathbb{S}^{N-1}} |\nabla_\tau \varphi|^p\,d\sigma+\int_{\mathbb{S}^{N-1}} |\varphi|^p\,d\sigma\right)^\frac{1}{p},
\]
for a constant $C=C(N,p)>0$. Here $d_{\mathbb{S}^{N-1}}$ is the geodesic distance on $\mathbb{S}^{N-1}$, computed with respect to the standard round metric, i.e. the Riemannian metric induced by the standard Euclidean one.
If we set
\[
\overline{\varphi}=\fint_{\mathbb{S}^{N-1}} \varphi\,d\sigma,
\] 
we can apply the previous inequality to the smooth function $\varphi-\overline{\varphi}$ and obtain
\[
[\varphi]_{C^{0,1-\frac{N-1}{p}}(\mathbb{S}^{N-1})}\le C\,\left(\int_{\mathbb{S}^{N-1}} |\nabla_\tau \varphi|^p\,d\sigma+\int_{\mathbb{S}^{N-1}} |\varphi-\overline{\varphi}|^p\,d\sigma\right)^\frac{1}{p}.
\]
We can estimate the rightmost term by means of the Poincar\'e-Wirtinger inequality on $\mathbb{S}^{N-1}$ (see \cite[Theorem 2.10]{He}), so to get
\[
[\varphi]_{C^{0,1-\frac{N-1}{p}}(\mathbb{S}^{N-1})}\le C\,\left(\int_{\mathbb{S}^{N-1}} |\nabla_\tau \varphi|^p\,d\sigma\right)^\frac{1}{p},
\]
possibly for a different constant $C=C(N,p)>0$. The previous Morrey--type inequality is valid for every $\varphi\in C^\infty(\mathbb{S}^{N-1})$. We now fix a point $x_0\in\mathbb{S}^{N-1}$ and use it for a smooth function vanishing at $x_0$. We get in particular
\[
|\varphi(x)|=|\varphi(x)-\varphi(x_0)|\le C\,d_{\mathbb{S}^{N-1}}(x,x_0)^{1-\frac{N-1}{p}}\,\left(\int_{\mathbb{S}^{N-1}} |\nabla_\tau \varphi|^p\,d\sigma\right)^\frac{1}{p}.
\]
By using that the geodesic distance on $\mathbb{S}^{N-1}$ is bounded, we can raise the previous estimate to the power $p$ and integrate over $\mathbb{S}^{N-1}$. This gives the desired conclusion.
\end{proof}

\subsection{Poincar\'e inequality Vs. inradius}

Finally, we conclude this section by giving an equivalent condition for {\rm (\ref{item:EOmega})} to hold, in the case of contractible sets. 
We recall the definition of inradius for an open set
\[
r_\Omega=\sup\Big\{r>0\, :\, \exists\, B_r(x_0)\subseteq \Omega\Big\}.
\]
We have that contractible open subsets of $\mathbb{R}^N$ supports the $L^p$ Poincar\'e inequality if and only if they have finite inradius, for $p$ large enough. This characterization is interesting in itself and will follow from the results of Appendix \ref{app:A}.
\begin{Proposition}
\label{prop:geometric}
Let $N-1<p<\infty$ and let $\Omega\subsetneq\mathbb{R}^N$ be an open contractible set. Then
\[
\lambda_p(\Omega)>0\qquad \Longleftrightarrow\qquad r_\Omega<+\infty.
\]
\end{Proposition}
\begin{proof}
The implication $\Longrightarrow$ is valid for every open set, as previously recalled in Proposition \ref{prop:inradius_finito}.
\par
For the converse implication, we first observe that the case $p>N$ is contained in \cite[Theorem 5.4]{BraPriZag}. In this case the result actually holds {\it for every open set}, without any further assumption. 
\par
We thus consider the case $N-1<p\le N$. By Propositions \ref{prop:corsimply} and \ref{prop:contratto}, we have that $\Omega\in\mathrm{UT}_p(c_{N,p})$. Then the desired conclusion follows directly from Theorem \ref{thm:salvaipotesi} below.
\end{proof}
\begin{Remark}
The case $N=2$ of the previous result was already contained in \cite[Theorem 1.2]{BozBra_JGA}, by means of a different proof. Actually, the latter provides the stronger result
\[
c_p\,\left(\frac{1}{r_\Omega}\right)^p\le \lambda_p(\Omega),
\]
which gives a {\it quantitative} version of the implication $\Longleftarrow$.
Observe that in the two-dimensional case, we can even allow the choice $p=N-1=1$. Finally, we recall that for $N\ge 3$ and $p\le N-1$, the equivalence stated in Proposition \ref{prop:geometric} cease to be true (see for example the Introduction of \cite{BozBra_JGA} for more details).
\end{Remark}

\section{Regularity for the Lane-Emden equation}
\label{sec:6}

We specialize the main result of the paper to the case of weak solutions of the equation
\begin{equation}
\label{LE}
-\Delta_p u=\lambda\,|u|^{m-2}\,u,\qquad \text{in}\ \Omega,
\end{equation}
for some $\lambda>0$ and $m\ge p$.
We will not take any further assumption on the open sets, apart for those of the Main Theorem. 
The following result will be useful, in order to get the desired regularity estimates in scale invariant form. 
\begin{Lemma}
\label{lm:riscaldo}
Let $1<p,m<\infty$, 
let $\Omega\subseteq\mathbb{R}^N$ be an open set and let $u\in W^{1,p}_0(\Omega)\cap L^m(\Omega)$  be a non-trivial weak solution of \eqref{LE}, for some $\lambda>0$. There exists $R>0$ such that
the scaled function 
\[
v(x):=\frac{R^{-\frac{N}{m}}}{\|u\|_{L^m(\Omega)}}\, u\left(\frac{x}{R}\right),\qquad \text{for}\ x\in R\,\Omega,
\] 
is a weak solution of \eqref{LE} in $R\,\Omega$, with $\lambda=1$. Moreover, we have
\[
\int_{R\,\Omega} |v|^m\,dx=1.
\]
\end{Lemma}
\begin{proof}
It is sufficient to define
\begin{equation}
\label{AR}
R=\left(\lambda^m\,\left(\int_\Omega |u|^m\,dx\right)^{m-p}\right)^\frac{1}{p\,m+N\,(p-m)}.
\end{equation}
Then the proof is by direct verification and thus omitted.
\end{proof}
We now come to the higher integrability result for solutions of \eqref{LE}.
For the sake of clarify, it will be convenient to distinguish the two cases: $m>p$ and $p=m$. We stress that we are here particularly interested in the precise form of the global a priori estimates.

\subsection{Super-homogeneous case}
We start with the case $m>p$.
\begin{Corollary}
\label{coro:super}
Let $1<p<m$, where $m$ verifies
\begin{equation}
\label{espo}
\left\{\begin{array}{rl}
m<p^*& \text{if}\ 1<p<N,\\
m<\infty,& \text{if}\ p\ge N.\\
\end{array}
\right.
\end{equation}
Let $\Omega\subseteq\mathbb{R}^N$ be an open set satisfying Assumptions \ref{ass:1} and \ref{ass:2}.
There exists an exponent $\delta=\delta(N,p,p_0,c_0,m)>0$ such that
every weak solution $u\in W_0^{1,p}(\Omega)$ of \eqref{LE} satisfies
\[
u\in W^{1,p\,(1+\delta)}_0(\Omega).
\]
Moreover,  by denoting with 
\begin{equation*}
\sigma:=\frac{p\,m+N\,(p-m)}{m},\qquad\qquad \xi:= 1+\frac{N}{m}-\frac{N}{p\,(1+\delta)},	
\end{equation*}
the following estimate holds
\[
\|\nabla u\|_{L^{p\,(1+\delta)}(\Omega)}\le C\,\lambda^\frac{\xi}{\sigma}\,\Big(\|u\|_{L^m(\Omega)}\Big)^{1+(m-p)\,\frac{\xi}{\sigma}},
\]
for a constant $C=C(N,p,p_0,c_0,m,\delta)>0$.
\end{Corollary}
\begin{proof}
Let $u\in W_0^{1,p}(\Omega)$ be a solution of \eqref{LE} and consider the scaled function 
\[
v(x)=\frac{R^{-\frac{N}{m}}}{\|u\|_{L^m(\Omega)}}\, u\left(\frac{x}{R}\right),\qquad \text{for}\ x\in R\,\Omega,
\] 
of Lemma \ref{lm:riscaldo}.
We now want to apply the Main Theorem to the function $v$, which is such that
\[
-\Delta_p v=|v|^{m-2}\,v,\quad \text{in}\ R\,\Omega,\qquad \|v\|_{L^m(R\,\Omega)}=1.
\] 
We thus apply it to the scaled open set $R\,\Omega$, with 
\[
f=|v|^{m-2}\,v,\qquad \gamma_0=m'\qquad\text{and}\qquad \gamma=+\infty.
\]
Indeed, observe that 
\[
\|f\|_{L^{m'}(R\,\Omega)}^{m'}= \int_{R\,\Omega} |v|^m\,dx=1\qquad\text{and}\qquad \|f\|_{L^\infty(R\,\Omega)}=\|v\|^{m-1}_{L^\infty(R\,\Omega)}\le C,
\] 
for a constant $C=C(N,p,m)>0$. The second fact follows from the $L^\infty$ estimate for solutions of the super-homogeneous Lane-Emden equation (see for example \cite[Proposition 2.4]{BraLin} and \cite[Lemma 2.3]{BraBriPri_steiner}) and the normalization conditions on $v$. 
\par
As for the needed assumptions on the open set $R\,\Omega$, we observe that, in view of Lemma \ref{lm:riscalothick}, its complement is still uniformly $p_0-$thick with the same constant $c_0$ as $\Omega$. 
Moreover, the assumption $\lambda_p(\Omega)>0$ implies that $\lambda_{p}(R\,\Omega)>0$, as well. 
Therefore, we can apply the Main Theorem and obtain existence of $\delta=\delta(N,p,p_0,c_0,m)>0$ such that $v\in W^{1,p\,(1+\delta)}_0(R\,\Omega)$, with the estimate
\begin{equation*}
\|\nabla v\|_{L^{p(1+\delta)}(R\,\Omega)} \le C_1,\qquad \text{where}\ C_1=C_1(N,p,p_0,c_0,\delta)>0.
\end{equation*}
By recalling the definition of $v$ and scaling back we obtain $u\in W^{1,p\,(1+\delta)}_0(\Omega)$, with the estimate
\[
\| \nabla u\|_{L^{p(1+\delta)}(\Omega)}\le C_1\, R^{1+\frac{N}{m}-\frac{N}{p\,(1+\delta)}}\,\|u\|_{L^m(\Omega)}.
\]
Finally, if we use the definition \eqref{AR} of $R$, we get the desired conclusion.
\end{proof}

\subsection{Eigenfunctions of the $p-$Laplacian}
In the homogeneous case $m=p$, the equation \eqref{LE} turns out to be 
\begin{equation}
\label{eigen}
-\Delta_p u=\lambda\,|u|^{p-2}\,u,\qquad \text{in}\ \Omega,
\end{equation}
i.e. the so-called eigenvalue equation for the $p-$Laplacian (see for example \cite{Lin}). We have the following result, which just follows by repeating the proof of Corollary \ref{coro:super} with $m=p$ and using the classical $L^\infty$ estimate for eigenfunctions
\[
\|u\|_{L^\infty(\Omega)}\le C_{N,p}\,\lambda^\frac{N}{p^2}\,\|u\|_{L^p(\Omega)}.
\]
\begin{Corollary}
\label{coro:eigen}
Let $1<p<\infty$ and let $\Omega\subseteq\mathbb{R}^N$ be an open set satisfying Assumptions \ref{ass:1} and \ref{ass:2}. 
There exists an exponent $\delta=\delta(N,p,p_0,c_0)>0$ such that
every weak solution $u\in W_0^{1,p}(\Omega)$ of \eqref{eigen} satisfies
\[
u\in W^{1,p\,(1+\delta)}_0(\Omega).
\]
Moreover, the following estimate holds
\[
\|\nabla u\|_{L^{p\,(1+\delta)}(\Omega)}\le C\,\lambda^\frac{p+\delta\,(N+p)}{p^2\,(1+\delta)}\,\|u\|_{L^p(\Omega)},
\]
for a constant $C=C(N,p,p_0,c_0,\delta)>0$.
\end{Corollary}
At this point, the Theorem for the eigenfunctions of the Dirichler-Laplacian, stated in the Introduction, simply follows from the previous results. More generally, we can obtain
\begin{Corollary}[Eigenfunctions in the conformal case]
\label{coro:original}
There exists a universal exponent $\delta=\delta(N)>0$ such that for every contractible open set $\Omega\subseteq\mathbb{R}^N$ with $r_\Omega<+\infty$ and every Dirichlet eigenfunction $u\in W^{1,N}_0(\Omega)\setminus\{0\}$ of the $N-$Laplacian with eigenvalue $\lambda$, we have 
\[
u\in W^{1,N\,(1+\delta)}_0(\Omega)\qquad \text{and}\qquad u\in C^{0,\frac{\delta}{1+\delta}}_0(\overline\Omega).
\]
Moreover, it holds
\[
\|\nabla u\|_{L^{N\,(1+\delta)}(\Omega)}\le C\,\left(\sqrt[N]{\lambda}\right)^\frac{1+2\,\delta}{1+\delta}\,\|u\|_{L^N(\Omega)},
\]
and
\[
|u|_{C^{0,\frac{\delta}{1+\delta}}(\overline{\Omega})}\le C\,\left(\sqrt[N]{\lambda}\right)^\frac{1+2\,\delta}{1+\delta}\,\|u\|_{L^N(\Omega)},
\]
for a universal constant $C=C(N)>0$.
\end{Corollary}
\begin{proof}
By using Propositions \ref{prop:corsimply}, \ref{prop:contratto} and \ref{prop:geometric}, we know that $\Omega$ satisfies Assumptions \ref{ass:1} and \ref{ass:2}. We can choose for example
\[
p_0=N-\frac{1}{2}.
\]
Accordingly, we have $\Omega\in \mathrm{UT}_{p_0}(c_0)$, with $c_0$ depending on the dimension $N$, only.
We can then apply Corollary \ref{coro:eigen} with $N=p$, so to get $u\in W^{1,N\,(1+\delta)}_0(\Omega)$ with the associated estimate. The global H\"older continuity follows from the Sobolev-Morrey-embedding. Observe that the relevant embedding constant only depends on $N$ and $\delta$ (and the latter in turn is only determined by $N$).
\end{proof}
\appendix

\section{Thick complement and Poincar\'e inequality}
\label{app:A}

In this appendix, we will prove some properties of open sets with uniform thick complement, which are interesting in themselves. 
\par
We start with a simple result, assuring that we can always enlarge the radius $r_0$ in the thickness condition, by possibly spoiling the constant $c_0$.
\begin{Lemma}
\label{lm:allarga_raggi}
Let $1\le p_0<\infty$  and let $\Omega \subsetneq  \mathbb{R}^N$ be an open set such that $\Omega\in \mathrm{UT}_{p_0}(c_0,r_0)$, for some $c_0>0$ and $r_0>0$.
There exists a constant $c=c(N,p_0)\in(0,1)$ such that for every $R_0>r_0$ we also have
\[
\Omega\in \mathrm{UT}_{p_0}\left( c\cdot c_0\left(\frac{r_0}{R_0}\right)^{N},R_0\right).
\]
In other words, for every $R_0>r_0$ we have 
\[
\mathrm{cap}_{p_0}\left(\overline{Q_r(x_0)}\setminus\Omega;Q_{2\,r}(x_0)\right)\ge c\cdot c_0\left(\frac{r_0}{R_0}\right)^{N}\,\mathrm{cap}_{p_0}\left(\overline{Q_r(x_0)};Q_{2\,r}(x_0)\right),\quad \text{for every}\ 0<r\le R_0,
\]
as well.
\end{Lemma}
\begin{proof}
We fix $R_0>r_0$ and consider a cube $Q_r(x_0)$ with $x_0\in\mathbb{R}^N\setminus\Omega$ and $r_0<r\le R_0$. By using the monotonicity properties of $p_0-$capacity, we have
\[
\begin{split}
\mathrm{cap}_{p_0}\left(\overline{Q_r(x_0)}\setminus\Omega;Q_{2\,r}(x_0)\right)&\ge \mathrm{cap}_{p_0}\left(\overline{Q_{r_0}(x_0)}\setminus\Omega;Q_{2R_0}(x_0)\right)\\
&=\frac{\mathrm{cap}_{p_0}\left(\overline{Q_{r_0}(x_0)}\setminus\Omega;Q_{2R_0}(x_0)\right)}{\mathrm{cap}_{p_0}\left(\overline{Q_r(x_0)};Q_{2\,r}(x_0)\right)}\,\mathrm{cap}_{p_0}\left(\overline{Q_r(x_0)};Q_{2\,r}(x_0)\right)\\
&=\frac{\mathrm{cap}_{p_0}\left(\overline{Q_{r_0}(x_0)}\setminus\Omega;Q_{2R_0}(x_0)\right)}{r^{N-p_0}\,\mathrm{cap}_{p_0}\left(\overline{Q_1};Q_{2}\right)}\,\mathrm{cap}_{p_0}\left(\overline{Q_r(x_0)};Q_{2\,r}(x_0)\right)\\
&\ge \frac{\mathrm{cap}_{p_0}\left(\overline{Q_{r_0}(x_0)}\setminus\Omega;Q_{2R_0}(x_0)\right)}{R_0^{N-p_0}\,\mathrm{cap}_{p_0}\left(\overline{Q_1};Q_{2}\right)}\,\mathrm{cap}_{p_0}\left(\overline{Q_r(x_0)};Q_{2\,r}(x_0)\right). 
\end{split}
\]
On the other hand, by \cite[Proposition 13.1.1.1]{Maz}, we can estimate 
\[
\mathrm{cap}_{p_0}\left(\overline{Q_{r_0}(x_0)}\setminus\Omega;Q_{2R_0}(x_0)\right)\ge c\,\left(\frac{r_0}{R_0}\right)^{p_0}\,\mathrm{cap}_{p_0}\left(\overline{Q_{r_0}(x_0)}\setminus\Omega;Q_{2r_0}(x_0)\right).
\]
A further application of the definition of $p_0-$thickness permits to bound from below the right-hand side. Thus, we get
\[
\mathrm{cap}_{p_0}\left(\overline{Q_r(x_0)}\setminus\Omega;Q_{2\,r}(x_0)\right)\ge c\cdot c_0\left(\frac{r_0}{R_0}\right)^{N}\,\mathrm{cap}_{p_0}\left(\overline{Q_r(x_0)};Q_{2\,r}(x_0)\right).
\]
This concludes the proof.
\end{proof}
The following result assures that the finiteness of the inradius, in conjunction with the thickness of the complement, is a sufficient condition for Poincar\'e inequality to hold. This is the main result of this appendix. Apart for its applications to our problem, it is interesting in itself.
\begin{Theorem}
\label{thm:salvaipotesi}
Let $1\le p_0<\infty$ and let $\Omega \subseteq  \mathbb{R}^N$ be an open set such that:
\begin{itemize}
\item $r_\Omega<+\infty$; 
\vskip.2cm
\item $\Omega\in \mathrm{UT}_{p_0}(c_0,r_0)$ for some $c_0>0$ and $r_0>0$.
\end{itemize}
Then
\[
\lambda_{p_0}(\Omega)>0\qquad \text{and}\qquad \Omega\in \mathrm{UT}_{p_0}(\widetilde{c_0}),
\] 
for some $\widetilde{c}_0=\widetilde{c}_0(N,p,\Omega)>0$.
\end{Theorem}
\begin{proof}
We divide the proof in two parts: we first prove that $\Omega$ supports the $L^{p_0}$ Poincar\'e inequality and then we show that the $p_0-$thickness condition holds uniformly for every $r>0$.
\vskip.2cm\noindent
{\bf Part 1: validity of the Poincar\'e inequality}. We claim that there exists a constant $C=C(N,p_0,r_0/r_\Omega)>0$ such that
\begin{equation}
\label{ammazzia!}
\mathrm{cap}_{p_0}\left(\overline{Q_{2r_\Omega}(x_0)}\setminus\Omega;Q_{4r_\Omega}(x_0)\right)\ge \frac{1}{C}\,r_\Omega^{N-p_0},\qquad \text{for every}\ x_0\in\mathbb{R}^N.
\end{equation}
In light of \cite[Theorem 15.4.2.1]{Maz}, this would be enough to get the desired conclusion.
In order to establish \eqref{ammazzia!}, we distinguish two cases: either $x_0\in\mathbb{R}^N\setminus\Omega$ or $x_0\in\Omega$.
\vskip.2cm\noindent
{\it Case $x_0\in\mathbb{R}^N\setminus\Omega$}. Our assumptions guarantee that there exists $\widetilde{c}_0>0$ such that for every $x_0\in\mathbb{R}^N\setminus\Omega$ we have
\begin{equation}
\label{premazya}
\mathrm{cap}_{p_0}\left(\overline{Q_r(x_0)}\setminus\Omega;Q_{2\,r}(x_0)\right)\ge \widetilde{c}_0\,\mathrm{cap}_{p_0}\left(\overline{Q_r(x_0)};Q_{2\,r}(x_0)\right),\quad \text{for every}\ 0<r\le 2\,r_\Omega.
\end{equation}
 Indeed, if $r_0\ge2\,r_\Omega$ this is immediate from the definition of $p_0-$thickness (and we can take $\widetilde{c}_0=c_0$). If $r_0<2\,r_\Omega$, we could apply Lemma \ref{lm:allarga_raggi} with the choice $R_0=2\,r_\Omega$ and get \eqref{premazya} with
 \[
\widetilde{c}_0= c\cdot c_0\left(\frac{r_0}{2\,r_\Omega}\right)^{N}.
 \]
In particular, from \eqref{premazya} with $r=2\,r_\Omega$ we easily get \eqref{ammazzia!} for $x_0\in\mathbb{R}^N\setminus\Omega$.
\vskip.2cm\noindent
{\it Case $x_0\in\Omega$}. We recall that 
\[
r_\Omega=\sup_{x\in\Omega} d_\Omega(x),\qquad \text{where}\quad d_\Omega(x)=\min_{y\in\partial\Omega} |x-y|,\qquad \text{for every}\ x\in\Omega.
\]
By observing that 
\[
\|x-y\|_{\ell^\infty}\le |x-y|,\qquad \text{for every}\ x,y\in\mathbb{R}^N,
\]
we thus obtain that 
\begin{equation}
\label{distanszia}
\mathfrak{d}(x):=\inf_{y\in\partial\Omega} \|x-y\|_{\ell^\infty}\le r_\Omega,\qquad \text{for every}\ x\in\Omega.
\end{equation}
Then, for every $x_0\in \Omega$ and every $\mathfrak{d}(x_0)<r\le 2\,r_\Omega+\mathfrak{d}(x_0)$, we have from Lemma \ref{lm:capacity_cubes}
\[
       \mathrm{cap}_{p_0}\left(\overline{Q_r(x_0)} \setminus \Omega; Q_{2r}(x_0)\right) \geq \frac{1}{C_2}\,\left(\frac{r-\mathfrak{d}(x_0)}{2\,r+\mathfrak{d}(x_0)}\right)^{p_0}\,(r - \mathfrak{d}(x_0))^{N-p_0}.
       \]
By choosing $r=2\,r_\Omega$ and using \eqref{distanszia}, we obtain in particular \eqref{ammazzia!} for $x_0\in\Omega$, as well.
\vskip.2cm\noindent
{\bf Part 2: universality of the radius}. We now need to show that there exists a constant $\widetilde{c}_0$ such that
\[
\mathrm{cap}_{p_0}\left(\overline{Q_r(x_0)}\setminus\Omega;Q_{2\,r}(x_0)\right)\ge \widetilde{c}_0\,\mathrm{cap}_{p_0}\left(\overline{Q_r(x_0)};Q_{2\,r}(x_0)\right),\quad \text{for every} \ x_0\in\mathbb{R}^N\setminus\Omega, \ r>0.
\]
More precisely, we are only left with verifying this property for $r>r_0$, since for $0<r\le r_0$ this is true by assumption.
\par
The case $p_0>N$ is easy, it is sufficient to observe that for every $x_0\in\mathbb{R}^N\setminus\Omega$ and every $r>0$, we have
\[
\begin{split}
\mathrm{cap}_{p_0}\left(\overline{Q_r(x_0)}\setminus\Omega;Q_{2\,r}(x_0)\right)&\ge \mathrm{cap}_{p_0}\left(\{x_0\};Q_{2\,r}(x_0)\right)\\
&=r^{N-p_0}\,\mathrm{cap}_{p_0}\left(\{x_0\};Q_{2}(x_0)\right)\\
&=\frac{\mathrm{cap}_{p_0}\left(\{x_0\};Q_{2}(x_0)\right)}{\mathrm{cap}_{p_0}\left(\overline{Q_1(x_0)};Q_{2}(x_0)\right)}\,r^{N-p_0}\,\mathrm{cap}_{p_0}\left(\overline{Q_1(x_0)};Q_{2}(x_0)\right)\\
&=:\frac{1}{C}\,\mathrm{cap}_{p_0}\left(\overline{Q_r(x_0)};Q_{2\,r}(x_0)\right).
\end{split}
\]
Observe that we used the scale property of the relative $p_0-$capacity and the fact that points have positive $p_0-$capacity, if $p_0>N$.
\par
Thus, let us consider the case $1\le p_0\le N$.
We need to recall at first the definition of {\it capacitary inradius of} $\Omega$, taken from \cite{BozBra1}. For every $0<\gamma<1$, this is defined by
\[
R_{p_0,\gamma}(\Omega):=\sup\Big\{r>0\, :\, \exists x_0\in\mathbb{R}^N\ \text{s.\,t.}\ \mathrm{cap}_{p_0}\left(\overline{B_r(x_0)}\setminus\Omega;B_{2r}(x_0)\right)\le \gamma\,\mathrm{cap}_{p_0}\left(\overline{B_r(x_0)};B_{2r}(x_0)\right)\Big\}.
\]
Then we have the following upper bound (see \cite[Main Theorem]{BozBra1})
\[
\lambda_{p_0}(\Omega)\le C_{N,p_0,\gamma}\,\left(\frac{1}{R_{p_0,\gamma}(\Omega)}\right)^{p_0},
\]
with the constant $C_{N,p_0,\gamma}$ which diverges to $+\infty$, as $\gamma$ goes to $1$. Observe that $\lambda_{p_0}(\Omega)>0$, by the first part of the proof.
\par
Moreover, we can also consider the following variant of the capacitary inradius, with cubes in place of balls
\[
R_{p_0,\gamma}(\Omega;Q_1):=\sup\Big\{r>0\, :\, \exists x_0\in\mathbb{R}^N\ \text{s.\,t.}\ \mathrm{cap}_{p_0}(\overline{Q_r(x_0)}\setminus\Omega;Q_{2r}(x_0))\le \gamma\,\mathrm{cap}_{p_0}\left(\overline{Q_r(x_0)};Q_{2r}(x_0)\right)\Big\}.
\]
This quantity has been introduced by Maz'ya, under the name of {\it inner cubic diameter of} $\Omega$, see \cite[Definition 10.2.2]{Maz85}. By \cite[Proposition 3.6]{BozBra2}, we have
\begin{equation}
\label{2}
R_{p_0,c\cdot\gamma}(\Omega;Q_1)\le R_{p_0,\gamma}(\Omega),
\end{equation}
for a constant $0<c<1$ depending only on $N$ and $p_0$. We now choose for simplicity $\gamma=1/2$, from \eqref{2} we get
\[
R_{p_0,c/2}(\Omega;Q_1)\le  R_{p_0,1/2}(\Omega)=:\mathfrak{r}_0.
\]
Thus, by the very definition of $R_{p_0,c\cdot\gamma}(\Omega;Q_1)$, we get that for every $r>\mathfrak{r}_0$
and for every $x_0\in\mathbb{R}^N$, we must have
\[
\mathrm{cap}_{p_0}(\overline{Q_r(x_0)}\setminus\Omega;Q_{2r}(x_0))> \frac{c}{2}\,\mathrm{cap}_{p_0}\left(\overline{Q_r(x_0)};Q_{2r}(x_0)\right).
\]
If $\mathfrak{r}_0\le r_0$, the proof is over. On the contrary, if $\mathfrak{r}_0>r_0$ we can simply apply Lemma \ref{lm:allarga_raggi} with $R_0=\mathfrak{r}_0$.
\end{proof}
\begin{Remark}[Sharpness of the assumptions]
\label{oss:UTcontro}
In the previous result, the assumption $r_\Omega<+\infty$ is crucial, i.e. in general its conclusion does not hold if the inradius is infinite. For example, by taking $1\le p_0<N$ and
\[
\Omega=\mathbb{R}^N\setminus \overline{Q_1},
\]
we see that its complement $\mathbb{R}^N\setminus\Omega=\overline{Q_1}$ is uniformly $p_0-$thick, for suitable $c_0>0$ and $0<r_0<+\infty$. Indeed, for every $x_0\in\overline{Q}_1$, we have
\[
\mathrm{cap}_{p_0}(\overline{Q_r(x_0)}\setminus\Omega;Q_{2r}(x_0))=\mathrm{cap}_{p_0}(\overline{Q_r(x_0)}\cap \overline{Q_1};Q_{2r}(x_0)).
\]
By the monotonicity properties of the $p_0-$capacity and the lower bound \cite[(2.2.12)]{Maz}, we can estimate from below the right-hand side as follows
\[
\begin{split}
\mathrm{cap}_{p_0}(\overline{Q_r(x_0)}\cap \overline{Q_1};Q_{2r}(x_0))&\ge \mathrm{cap}_{p_0}(\overline{Q_r(x_0)}\cap \overline{Q_1};\mathbb{R}^N)\\
&\ge(N\,\omega_N)^\frac{p_0}{N}\,N^\frac{N-p_0}{N}\,\left(\frac{N-p_0}{p_0-1}\right)^{p_0-1}\,\left|\overline{Q_r(x_0)}\cap \overline{Q_1}\right|^\frac{N-p_0}{N}.
\end{split}
\]
We now observe that 
\[
\left|\overline{Q_r(x_0)}\cap \overline{Q_1}\right|\ge \frac{1}{2^N}\,|Q_r(x_0)|=r^N,\qquad \text{for every}\ x_0\in\overline{Q}_1,\ 0<r\le1.
\]
Thus, we conclude that $\Omega\in \mathrm{UT}_{p_0}(c_0,r_0)$, with
\[
c_0=\frac{(N\,\omega_N)^\frac{p_0}{N}\,N^\frac{N-p_0}{N}\,\left(\frac{N-p_0}{p_0-1}\right)^{p_0-1}}{\mathrm{cap}_{p_0}(\overline{Q_1};Q_2)}\qquad \text{and}\qquad r_0=1.
\]
On the other hand, for this set $\Omega$ we have 
\[
\lambda_{p_0}(\Omega)=0\qquad \text{and}\qquad \Omega\not\in \bigcup_{c>0}\mathrm{UT}_{p_0}(c\,;+\infty). 
\]
The first fact follows since $\Omega$ contains arbitrarily large balls. The second one follows
since
\begin{equation}
\label{azzeta}
\lim_{r\to +\infty} \frac{\mathrm{cap}_{p_0}(\overline{Q_r}\setminus\Omega;Q_{2r})}{\mathrm{cap}_{p_0}(\overline{Q_r};Q_{2r})}=0.
\end{equation}
It is sufficient to observe that 
\[
\overline{Q_r}\setminus\Omega=\overline{Q_1},\qquad \text{for every}\ r\ge 1.
\]
Then, by taking $\eta\in C^\infty_0(Q_{2})$ such that $\eta\ge 1$ on $\overline{Q_1}$, we get for every $r\ge 1$
\[
\frac{\mathrm{cap}_{p_0}(\overline{Q_r}\setminus\Omega;Q_{2r})}{\mathrm{cap}_{p_0}(\overline{Q_r};Q_{2r})}=\frac{\mathrm{cap}_{p_0}(\overline{Q_1};Q_{2r})}{\mathrm{cap}_{p_0}(\overline{Q_r};Q_{2r})}\le \frac{\displaystyle\int_{Q_2}|\nabla \eta|^{p_0}\,dx}{r^{N-p_0}}\,\frac{1}{\mathrm{cap}_{p_0}(\overline{Q_1};Q_2)}.
\]
By taking the limit as $r$ goes to $\infty$ and using that $p_0<N$, we get \eqref{azzeta}.  
\end{Remark}
From Theorem \ref{thm:salvaipotesi}, we immediately get the following consequence.
\begin{Corollary}
Let $1\le p_0<\infty$ and let $\Omega \subseteq  \mathbb{R}^N$ be an open set such that:
\begin{itemize}
\item $r_\Omega<+\infty$; 
\vskip.2cm
\item $\Omega\in \mathrm{UT}_{p_0}(c_0,r_0)$ for some $c_0>0$ and $r_0>0$.
\end{itemize}
Then $\lambda_p(\Omega)>0$, for every $p\ge p_0$. 
\end{Corollary}
\begin{proof}
Once we have obtained $\lambda_{p_0}(\Omega)>0$ from Theorem \ref{thm:salvaipotesi}, the desired conclusion follows from the inequality
\[
\left(\frac{p_0}{p}\right)^p\,\big(\lambda_{p_0}(\Omega)\big)^\frac{p}{p_0}\le \lambda_p(\Omega),
\]
see for example \cite[Proposition 2.1]{BBV}.
\end{proof}

\section{A note on the case $1<p<N$}
\label{app:B}

In the case $1<p<N$, if we allow $\gamma_0=(p^*)'$ in {\rm (\ref{item:Ef})} of Assumptions \ref{ass:1}, one could be dispensed with the assumption {\rm (\ref{item:EOmega})}. 
However, the natural functional space where the solution lies is now the homogeneous Sobolev space $\mathscr{D}^{1,p}_0(\Omega)$. 
Accordingly, the statement of the Main Theorem should be amended as follows. The proof would be just a minor variant of that of the Main Theorem: therefore, it is omitted.
\begin{Theorem}[Exceptional case]
 \label{maintheorem}
 Let $1<p<N$ and let $\Omega\subseteq\mathbb{R}^N$ be an open set such that
\[
\Omega\in \mathrm{UT}_{p_0}(c_0,r_0),\qquad \text{for some}\ 1\le p_0<p\quad \text{and}\quad c_0,r_0>0.
\]
Let $f\in L^{(p^*)'}(\Omega)\cap L^\gamma(\Omega)$, for some $(p^*)'<\gamma\le \infty$. Finally, 
let $u\in \mathscr{D}_0^{1,p}(\Omega)$ be the weak solution to the equation \eqref{problem}, where $G$ still satisfies the assumption {\rm (\ref{item:EG})}.
\par
Then, there exists an exponent $\delta=\delta(N,p,p_0,\gamma,c_0,g_1,g_2)>0$ with $\delta\le (\gamma-(p^*)')/(p^*)'$,
such that 
\[
\nabla u\in L^{p\,(1+\delta)}(\Omega).
\] 
Moreover, there exists a sequence $\{\varphi_n\}_{n\in\mathbb{N}} \subseteq C^\infty_0(\Omega)$ such that
\[
\lim_{n\to\infty} \|\nabla u-\nabla\varphi_n\|_{L^{p\,(1+\delta)}(\Omega)}=0.
\]
Finally, the following scale invariant estimate holds, for $C_1=C_1(N,p,p_0,\gamma,c_0,g_1,g_2)>0$ 
\[
\begin{split}
\left(\int_\Omega |\nabla u|^{p\,(1+\delta)}\,dx\right)^\frac{1}{1+\delta}&\le \frac{C_1}{r_0^{N\frac{\delta}{1+\delta}}}\,\int_\Omega |\nabla u|^p\,dx\\
&+C_1\,\left(\|f\|^{\gamma_0}_{L^{\gamma_0}(\Omega)}\right)^{\frac{p'}{\gamma_0}-\frac{\delta}{1+\delta}\,\frac{\gamma}{\gamma-\gamma_0}}\,\left(\|f\|^\gamma_{L^{\gamma}(\Omega)}\right)^{\frac{\delta}{1+\delta}\,\frac{\gamma_0}{\gamma-\gamma_0}}.
\end{split}
\]
It is intended that the first term on the right-hand side has to be considered as zero, in the case $r_0=+\infty$.
\end{Theorem}

\begin{Remark}
We remark that, under the previous assumptions, in general we could have $r_0<+\infty$. It is sufficient to keep in mind the example of Remark \ref{oss:UTcontro}.
\end{Remark}

\end{document}